\documentclass[12pt,a4paper]{article}
\usepackage{amsmath,amssymb,amsfonts,amsthm}
\usepackage{float,graphicx,color}

\newtheorem{theorem}{Theorem}[section]
\newtheorem{lemma}[theorem]{Lemma}
\newtheorem{proposition}[theorem]{Proposition}
\newtheorem{corollary}[theorem]{Corollary}

\theoremstyle{definition}

\newtheorem{construction}[theorem]{Construction}
\newtheorem{example}[theorem]{Example}
\newtheorem{remark}[theorem]{Remark}

\newcommand{\PGammaL}{\mathop{\mathrm{P}\Gamma\mathrm{L}}}
\newcommand{\PO}{\mathop{\mathrm{PO}}}
\newcommand{\PGL}{\mathop{\mathrm{PGL}}}

\newcommand{\AGL}{\mathop{\mathrm{AGL}}}
\newcommand{\AG}{\mathop{\mathrm{AG}}}
\newcommand{\PSL}{\mathop{\mathrm{PSL}}}

\newcommand{\PSp}{\mathop{\mathrm{PSp}}}
\newcommand{\PSigmaL}{\mathop{\mathrm{P}\Sigma\mathrm{L}}}
\newcommand{\PGammaSp}{\mathop{\mathrm{P}\Gamma\mathrm{Sp}}}
\newcommand{\GF}{\mathop{\mathrm{GF}}}
\DeclareMathOperator{\Cos}{Cos}
\newcommand{\POmega}{\mathop{\mathrm{P}\Omega}}
\DeclareMathOperator{\Wr}{wr}
\newcommand{\norml}{\vartriangleleft}

\newcommand{\F}{{\ensuremath{\mathbb{F}}}}

\DeclareMathOperator{\Aut}{Aut}
\DeclareMathOperator{\Inn}{Inn}
\DeclareMathOperator{\Sym}{Sym}
\newcommand{\la}{\langle}
\newcommand{\ra}{\rangle}
\DeclareMathOperator{\soc}{soc}

\title{On finite edge-primitive and edge-quasiprimitive graphs\thanks{During the prepartion of this work the first author held an Australian Research Council Australian Postdoctoral Fellowship while the second author held a QEII Fellowship.}}
\author{Michael Giudici and Cai Heng Li \\
       School of Mathematics and Statistics\\
       The University of Western Australia\\
       35 Stirling Highway\\
       Crawley WA 6009\\
       Australia}
\date{}

\begin{document}
\maketitle

\begin{abstract}
Many famous graphs are edge-primitive, for example, the Heawood graph,
the Tutte--Coxeter graph and the Higman--Sims graph. In this paper we
systematically analyse edge-primitive and edge-quasiprimitive graphs via the O'Nan--Scott Theorem to
determine the possible edge and vertex actions of such graphs. Many
interesting examples are given and we also determine all
$G$-edge-primitive graphs for $G$ an almost simple group with socle
$\PSL(2,q)$.
\end{abstract}

\section{Introduction}
Let $\Gamma$ be a finite connected graph and $G\leqslant\Aut(\Gamma)$.
We say that $\Gamma$ is \emph{$G$-edge-primitive} if $G$ acts
primitively on the set of edges of $\Gamma$, that is, if $G$ preserves
no nontrivial partition of the edge set. If $\Gamma$ is
$\Aut(\Gamma)$-edge-primitive we call $\Gamma$ \emph{edge-primitive}.
The aim of this paper is to initiate a systematic study of
edge-primitive graphs and the wider class of
\emph{edge-quasiprimitive} graphs, that is graphs with a
group of automorphisms which acts quasiprimitively on edges. (A
transitive permutation group is said to be \emph{quasiprimitive} if
every nontrivial normal subgroup is transitive). 

The Atlas \cite{atlas} notes many
edge-primitive graphs with a sporadic simple group as a group of
automorphisms. These include the Hoffman--Singleton and Higman--Sims
graphs, and the rank three graphs of the sporadic simple groups $J_2$, $McL$, $Ru$, $Suz$ and $Fi_{23}$. 
Weiss \cite{Weissprim} has determined all edge-primitive graphs of
valency three. These are the complete bipartite graph $K_{3,3}$, the
Heawood graph, the Biggs--Smith cubic distance-transitive graph on 102
vertices and the Tutte--Coxeter graph (also known as Tutte's 8-cage or
the Levi graph). All but the Biggs--Smith graph are
bipartite. We say that $\Gamma$ is \emph{$s$-arc-transitive} if the
automorphism group acts transitively on the set of \emph{$s$-arcs} of
$\Gamma$, that is, on the set of $(s+1)$-tuples
$(v_0,v_1,\ldots,v_s)$ where $v_i$ is adjacent to $v_{i+1}$ and
$v_i\neq v_{i+2}$. Of the four edge-primitive cubic graphs, $K_{3,3}$
is 3-arc-transitive, the next two are 4-arc-transitive while the
Tutte--Coxeter graph is 5-arc-transitive. 

Whereas any primitive permutation group with a nontrivial
self-paired orbital gives rise to a vertex-primitive graph, the
existence of edge-primitive graphs is far more restrictive. Given a
group $G$ there is a $G$-edge-primitive 
graph if and only if there exists a maximal subgroup $E$ of $G$
with an index two subgroup properly contained in
some corefree subgroup $H$ of $G$ with $H\neq E$ (see 
Proposition \ref{prn:general} and Lemma \ref{lem:arctrans}). 

One of the main motivations for our study of edge-primitive and
edge-quasiprimitive graphs is the study of graph decompositions
\cite{generaldecomps}.
Given a graph $\Gamma$ and a group of automorphisms
$G$, we say that a partition $\mathcal{P}$ of the edge set is a 
\emph{$G$-transitive decomposition} if $\mathcal{P}$ is $G$-invariant
and $G$ acts transitively on $\mathcal{P}$. A $G$-transitive decomposition
$\mathcal{P}$ of a graph $\Gamma$ is called a \emph{homogeneous
  factorisation} if the kernel of the action of $G$ on
$\mathcal{P}$ is vertex-transitive. Homogeneous factorisations have
been studied in \cite{GLPP2,GLPP1,LP03}. Let $\Gamma$ be a
$G$-edge-transitive graph. Then $\Gamma$ is $G$-edge-primitive if and
only if $\Gamma$ has no $G$-transitive decompositions.  
If $G$ is edge-quasiprimitive then
the $G$-transitive decompositions of $\Gamma$ are not homogeneous
factorisations. Conversely, if none of the $G$-transitive
decompositions of $\Gamma$ are homogeneous factorisations then the
kernel of each $G$-transitive decomposition is vertex-intransitive.  

If $\Gamma$ is a bipartite graph with a
vertex-transitive group of automorphisms $G$, then $G$ has a normal
subgroup $G^+$ of index 
two which fixes each of the bipartite halves setwise. We say that a
transitive group $G$ is \emph{biprimitive} if it is imprimitive and
all nontrivial systems of imprimitivity have precisely two parts,
while we say that $G$ is \emph{biquasiprimitive} if $G$ is not
quasiprimitive and every normal subgroup has at most two orbits. We
note here that some authors' definition of biprimitive as a transitive
permutation group $G$ with index two subgroup $G^+$ acting primitively
on both of its orbits is not equivalent to ours. All our biprimitive
groups are biprimitive in this sense but not all biprimitive groups in
this alternative sense are biquasiprimitive. For example
$S_n\times S_2$ acting imprimitively on $2n$ points for $n\geq 3$ has
a system of imprimitivity with $n$ parts of size $2$ while the index
two subgroup $S_n$ acts primitively on each of its orbits. 
Given property $P$, we say that a graph $\Gamma$ with a group of
automorphisms $G$ is \emph{$G$-locally $P$} if for each
vertex $v$, the vertex stabiliser $G_v$ has property $P$ on the set
$\Gamma(v)$ of all vertices adjacent to $v$. In particular, $\Gamma$
is called \emph{$G$-locally primitive} is $G_v$ acts primitively on
$\Gamma(v)$ for all vertices $v$.

For any positive integer $n$ and prime $p$, the star $K_{1,n}$ and the cycle $C_p$ are both
edge-primitive. We call these two
examples \emph{trivial}. Disconnected edge-primitive graphs are easily
reduced to connected ones (see Lemma \ref{lem:con}). We see in
Lemma \ref{lem:arctrans} that except for the trivial examples,
edge-primitivity implies arc-transitivity. 

Let $\Gamma$ be a connected $G$-arc-transitive graph and let
$\mathcal{B}$ be a $G$-invariant partition of $V\Gamma$. We define the
\emph{quotient graph} $\Gamma_{\mathcal{B}}$ to be the graph with
vertex set $\mathcal{B}$ such that $B,C\in\mathcal{B}$ are adjacent if
and only if $\Gamma$ has an edge $\{v,w\}$ with $v\in B$ and $w\in C$.  
It easily follows that $\Gamma_{\mathcal{B}}$ is arc-transitive. We
are interested in the special case where for an arc $(B,C)$ of
$\Gamma_{\mathcal{B}}$, there is only one arc $(v,w)$ of $\Gamma$ with
$v\in B$ and $w\in C$. In this case we call $\Gamma$ a
\emph{spread} of $\Gamma_{\mathcal{B}}$. 

We will see in Lemma \ref{lem:qpbiqp} that if $G$ is 
edge-primitive and vertex-transitive then it is either
vertex-quasiprimitive or vertex-biquasiprimitive on
vertices. In fact we can reduce to the vertex-primitive or
vertex-biprimitive cases.

\begin{theorem}
\label{thm:main}
Let $\Gamma$ be a connected nontrivial $G$-edge-primitive graph. Then
$\Gamma$ is $G$-arc-transitive, and one of the following holds.
\begin{enumerate}
\item $\Gamma$ is $G$-vertex-primitive.
\item $\Gamma$ is $G$-vertex-biprimitive.
\item  $\Gamma$ is a spread of a $G$-edge-primitive graph which is
  $G$-locally imprimitive.
\end{enumerate}
Conversely, a $G$-edge-primitive, $G$-locally imprimitive graph
$\Sigma$ is a quotient graph of a larger
$G$-edge-primitive graph $\Gamma$ with $G^{E\Sigma}\cong G^{E\Gamma}$.
\end{theorem}

This reduces the study of edge-primitive graphs to those which are
also vertex-primitive or vertex-biprimitive. 

The actions of primitive permutation groups are described by the
O'Nan--Scott Theorem. We follow the subdivision in \cite{praegerbcc}
of primitive groups 
into 8 types and these are described in Section \ref{sec:prelim}. By
playing the edge-primitive action of $G$ against the vertex-primitive
action of $G$ or $G^+$ we see that the possible actions for
edge-primitive graphs are quite restrictive.

\begin{theorem}
\label{thm:eprim}
Let $\Gamma$ be a connected nontrivial $G$-edge-primitive graph with
$G^{E\Gamma}$ primitive of type $X$ such that $G^{V\Gamma}$ is either primitive or biprimitive. Then one of the following holds.
\begin{enumerate}
\item $\Gamma=K_{n,n}$.
\item $G^{V\Gamma}$ is primitive of type $X$ and 
$X\in\{\mathrm{AS, PA}\}$.
\item $G^{V\Gamma}$ is biprimitive and $G^+$ is primitive of type $X$
  on each orbit with $X\in\{\mathrm{AS,PA}\}$.
\item $G^{E\Gamma}$ is of type {\rm SD} or {\rm CD}, $\Gamma$ is
  bipartite and arises from Construction \ref{con:SDCDtrans}, and
  $G^+$ is primitive of type {\rm CD} on each orbit. 
\end{enumerate}
\end{theorem}

We see in Sections \ref{sec:eg} and \ref{sec:con} that examples
exist in all cases. Moreover, we can find $G$-locally
imprimitive examples in each case. A characterisation of all groups
which act edge-primitively on $K_{n,n}$ is given in Theorem
\ref{thm:compbipartite}. We also see in Proposition
\ref{prn:PAtoAS} that the existence of $G$-edge-primitive graphs
with $G$ of type PA relies on the existence of edge-primitive graphs
where the action on edges is of type AS.  

We undertake much of our analysis in the context of
vertex-quasiprimitive graphs and only specialise to the edge-primitive
case when we are able to obtain stronger conclusions. There are
however, a couple of notable differences between the two classes.
There are many 
$G$-edge-quasiprimitive graphs with $G$ not vertex-transitive,
for example any bipartite graph with an edge-transitive simple group
$G$ of automorphisms is $G$-edge-quasiprimitive while $G$ has two
orbits on vertices. Vertex-transitive, edge-quasiprimitive graphs are
still either vertex-quasiprimitive or vertex-biquasiprimitive but we
are no longer able to reduce to the vertex-primitive or
vertex-biprimitive cases. Theorem \ref{thm:edgeqp} is an analogue of
Theorem \ref{thm:eprim} in the $G$-vertex-transitive,
$G$-edge-quasiprimitive case.

It appears feasible to determine all edge-primitive graphs for certain
families of almost simple groups, for example, for low rank groups of
Lie type. We begin this process in Section
\ref{sec:PSL} by determining all $G$-edge-primitive graphs where
$\soc(G)=\PSL(2,q)$. The \emph{socle} (denoted $\soc(G)$) of a group
$G$ is the product of all of its minimal normal subgroups.

\begin{theorem} 
\label{thm:PSL}
Let $\Gamma$ be a $G$-edge-primitive graph with
$\soc(G)=\PSL(2,q)$, such that $q=p^f$ for some prime $p$ and 
$q\neq 2,3$. Then either $\Gamma$ is 
complete and $G$ is listed in Table \ref{tab:2subsetprim}, or 
$\Gamma$ and $G$ are given in Table \ref{tab:psl}.
\end{theorem}

\begin{table}
\begin{center}
\caption{$G$-edge primitive with $\soc(G)=\PSL(2,q)$}
\label{tab:psl}
\begin{tabular}{|l|l|}
\hline
$G$ & $\Gamma$  \\
\hline\hline
$\PGL(2,7)$ & Heawood graph (Example \ref{eg:pthyp})\\ 
$\PGL(2,7)$  & co-Heawood graph (Example \ref{eg:pthyp})\\
$\PGL(2,9)$, $M_{10}$ or $\PGammaL(2,9)$  & $K_{6,6}$ \\
$\PGL(2,9)$, $M_{10}$ or $\PGammaL(2,9)$ & Tutte--Coxeter graph (Example \ref{eg:symp4}) \\
$\PGL(2,11)$  & $(H,E,E\cap H)=(A_5,D_{20}, D_{10})$ \\
$\PSL(2,17)$ & Biggs--Smith graph \\
             & $(H,E,E\cap H)=(S_4,D_{16}, D_8)$ \\
$\PSL(2,19)$ &$(H,E,E\cap H)=(A_5,D_{20}, D_{10} )$ \\
$\PSL(2,25)$ or $\PSigmaL(2,25)$ & Example \ref{eg:orthognoniso} \\
$\PSL(2,p)$, $p\equiv \pm 1,\pm 9 \pmod {40}$  & $(H,E,E\cap H)=(A_5,S_4,A_4)$  \\
$\PGL(2,p)$, $p\equiv \pm 11,\pm 19\pmod{40}$ &$(H,E,E\cap H)=(A_5,S_4,A_4)$  \\
\hline
\end{tabular}
\end{center}
\end{table}
In some rows of Table \ref{tab:psl} we just state the edge stabiliser
$E$ and vertex stabiliser $H$ along with $H\cap E$ as by Proposition
\ref{prn:general}, a $G$-edge-transitive graph is uniquely determined by the 
vertex stabiliser and edge stabiliser. Note for the first two examples
$\PGL(2,7)\cong\Aut(\PSL(3,2))$, for the fourth example note 
$\PGammaL(2,9)\cong \Aut(\PSp(4,2))$, while for the eighth example
$\PSL(2,25)\cong \POmega^-(4,5)$. Apart from complete graphs and
$K_{6,6}$, we get two infinite families and seven sporadic
examples. All of the graphs listed in Table \ref{tab:psl} are
2-arc-transitive except for the eighth one.

\section{Some examples}
\label{sec:eg}
If $G\leqslant S_n$ acts arc-transitively on $K_n$ then $G$ is
2-transitive on vertices.
Moreover, $G$ is edge-primitive if and only if $G$ acts primitively on
2-subsets. The following theorem, which is essentially 
\cite[Theorem 6]{sibley}, classifies all such $G$. 

\begin{theorem}
\label{thm:prim2subsets}
Let $G$ be a 2-transitive subgroup of $S_n$ such that $G$ is primitive
on 2-subsets. Then $G$ and $n$ are as in Table \ref{tab:2subsetprim}.
\end{theorem}
\begin{proof}
 By Burnside's Theorem (see for example \cite[Theorem
  4.1B]{dixon}), $G$ is either almost simple or a subgroup of
 $\AGL(d,p)$ with $n=p^d$ for some prime $p$. Sibley \cite{sibley}
 classified all $G$-transitive decompositions of $K_n$ for $G$ a
 2-transitive simple group and so this yields a classification of almost simple groups acting
edge-primitively on $K_n$.  Suppose now that $G\leqslant \AGL(d,p)$
 and let $u,v$ be a pair of points of $\AG(d,p)$. Then $\{u,v\}$ lies
 on a unique line $l$ and so $G_{\{u,v\}}\leqslant G_l< G_B\leqslant
 G$, where $B$ is the parallel class containing $l$. Thus for $d\geq
 2$, $G$ is not primitive on 2-subsets. Note that this includes $A_4$
 and $S_4$. When $d=1$, there is a unique parallel class and
 $G_{\{u,v\}}\cong C_2$. In this case, $G$ is primitive on 2-subsets
 if and only if $p=2$ or $3$. Here $G\cong S_2,S_3$ respectively.  
\end{proof}
\begin{table}
\begin{center}
\caption{2-transitive groups which are primitive on 2-subsets}
\label{tab:2subsetprim}
\begin{tabular}{|lll|}
\hline
$n$ & $G$  & Conditions\\
\hline\hline
$n$     &$S_n$ & $n\neq 4$\\
$n$     &$A_n$ & $n\geq 5$\\
$q+1$   & $\soc(G)=\PSL(2,q)$& $q\geq 7$ \\
&&$\begin{array}{ll} 
   \!\!\! G\neq &\!\! \PSL(2,7), \,\, \PSL(2,9), \\
            &\!\! \PSigmaL(2,9) \text{ or } \PSL(2,11). 
                  \end{array} $\\
$q^2+1$ &$\soc(G)=Sz(q)$ & $q=2^{2d+1}$\\
11      &$\PSL(2,11)$ & \\
11      & $M_{11}$ &\\
12      & $M_{11}$ &\\
12      & $M_{12}$ &\\
22      & $M_{22}$, $\Aut(M_{22})$&\\
23      & $M_{23}$ &\\
24       & $M_{24}$ &\\
176      & $HS$ &\\
276     & $Co_3$ &\\
\hline
\end{tabular}

\end{center}
\end{table}

 There are many geometrical constructions  of edge-primitive graphs with the following being just a couple.

\begin{example}
\label{eg:pthyp}
{\rm 
Let $T=\PSL(d,q)$ for $d\geq 3$ and $G=\Aut(T)$. Let $\Delta_1$ be the
set of $r$-dimensional subspaces of a $d$--dimensional vector space
over $\GF(q)$ with $1\leq r <d/2$ and let $\Delta_2$ be the set of
$(d-r)$-dimensional  subspaces. We define
$\Gamma$ to be the bipartite graph with vertex set $\Delta_1\cup \Delta_2$ with adjacency given by inclusion. Then $G\leqslant\Aut(\Gamma)$ and acts
biprimitively on vertices such that the stabiliser $G^+$ of the
bipartition is equal to $\PGammaL(d,q)$.  Moreover, the
stabiliser $E$ of an edge is a maximal subgroup of $G$ and so $\Gamma$ is
$G$-edge-primitive. When $(d,r)=(3,1)$, the graph obtained is 4-arc
transitive and when $(d,r,q)=(3,1,2)$, the graph obtained is the
Heawood graph.   

Alternatively, we can define an $r$-space to be adjacent to a
$(d-r)$-space if they are complementary. This also gives us a
$G$-edge-primitive graph with $G$ acting biprimitively on vertices and
when $(d,r,q)=(3,1,2)$ we get the co-Heawood graph.
}%
\end{example}
\begin{example}
\label{eg:symp4}
{\rm Let $V$ be a 4-dimensional vector space over $\GF(q)$ with $q$ even
and let $B$ be a nondegenerate alternating form. Let $\Delta_1$ be the
set of totally isotropic 1-spaces and $\Delta_2$ be the set of totally
isotropic 2-spaces. Define $\Gamma$ to be the graph with vertex set
$\Delta_1\cup\Delta_2$ and adjacency defined by inclusion. Then
$\PGammaSp(4,q)$ is an edge-transitive group of automorphisms of $\Gamma$
but has two orbits on vertices. Let $\tau$ be a duality of the polar
space interchanging $\Delta_1$ and $\Delta_2$. Then 
$G=\la \PGammaSp(4,q),\tau\ra$ is an arc-transitive group
of automorphisms of $\Gamma$ which is vertex-biprimitive. Moreover,
an edge stabiliser $G_e$ is a maximal subgroup of
$G$ and so $\Gamma$ is $G$-edge-primitive. When $q=2$, $\Gamma$ is the
Tutte--Coxeter graph. 
}%
\end{example}

There are also many other geometrical constructions of infinite families of
edge-primitive graphs involving sesquilinear or quadratic forms. We give one such
example here.

\begin{example}
\label{eg:orthognoniso}
{\rm Let $V$ be a vector space of dimension $d$ over the field $\GF(q)$,
with $q=3$ or $5$, and let $Q$ be a nondegenerate quadratic form on $V$ with
associated bilinear form $B$. Let $\Gamma$ be the graph whose vertex
set is the set of all nonsingular 1-spaces upon which the quadratic
form is a square with adjacency given by orthogonality with respect
to $B$. By Witt's Lemma, the group $G=\PO(d,q)$ of all isometries of
$Q$ is an arc-transitive automorphism group of $\Gamma$. 

Let $e=\{\la v\ra,\la w\ra\}$ be an edge of $\Gamma$. If
$q=5$ then $\la v,w\ra$ is a hyperbolic line while if $q=3$ then 
$\la v,w\ra$ is anisotropic. Moreover, in both cases $\la v\ra$, 
$\la w\ra$ are the only 1-spaces of $\la v,w\ra$ upon which $Q$ is a
square. Thus $G_e=G_{\la v,w\ra}$. By \cite{Kingnonisosubs}, it
follows that if $q=5$ then $G_e$ is maximal in $G$ except when $d=4$
and $Q$ is hyperbolic. Also, if $q=3$ then $G_e$ is maximal in $G$
except when $d=4$ or $5$.  
}%
\end{example}

Edge-primitive graphs can be defined via group theoretic means using
the coset graph construction.
Let $G$ be a group with a core-free subgroup $H$. Let $g\in G$ such
that $g$ does not normalise $H$ and $g^2\in H$. We define the coset
graph $\Gamma=\Cos(G,H,HgH)$ to have vertex set, the set $[G:H]$ of
right cosets of $H$ in $G$ with two vertices $Hx, Hy$ being adjacent if
and only if $xy^{-1}\in HgH$. The graph $\Gamma$ is connected if and
only if $\la H,g\ra= G$. Moreover, $G$ acts as an arc-transitive group
of automorphisms of $\Gamma$ via right multiplication. The valency of
$\Gamma$ is $|H:H\cap H^g|$ while the stabiliser of the edge $\{H,Hg\}$ is
$\la H\cap H^g,g\ra$. Conversely, suppose that $\Gamma$ is a
graph with adjacent vertices $v$ and $w$. Let $G\leqslant\Aut(\Gamma)$
be arc-transitive and let $g\in G$ interchange $v$ and $w$. Then
$\Gamma\cong \Cos(G,G_v,G_vgG_v)$. We have the following
characterisation of arc-transitive edge-primitive graphs.

\begin{proposition}
\label{prn:general}
Let $G$ be a group with a maximal subgroup $E$. Then there exists a
$G$-edge-primitive, arc-transitive graph  $\Gamma$ with edge
stabiliser $E$ if and only if $E$ has a subgroup $A$ of index two, and
$G$ has a corefree subgroup $H$ such that $A<H\neq E$; in this case
$\Gamma=\Cos(G,H,HgH)$ for some $g\in E\backslash A$.
\end{proposition}
\begin{proof}
Suppose first that $G,E,A,H$ and $g$ are as in the statement.
Since $E$ is maximal in $G$ and $H$ is not contained in $E$ it follows
that $E<\la H,g\ra =G$. As $H$ is corefree in $G$ we have that $g$
does not normalise $H$. Let $\Gamma=\Cos(G,H,HgH)$, let $v=H$, $w=Hg$
and $e=\{v,w\}$. Then $\Gamma$ is connected, $G_v=H$, $G_w=H^g$,
$G_{vw}=H\cap H^g$ and $G_{e}=\la H\cap H^g,g\ra\neq G$.
 Since $g$ does not normalise $H$, but does normalise $A$ we have
$A\leqslant H\cap H^g<H$ and so 
$E\leqslant \la H\cap H^g,g\ra=G_e$. The maximality of $E$ implies that
$G_e=E$ and $\Gamma$ is edge-primitive.

Conversely, suppose that $\Gamma$ is a $G$-arc-transitive,
$G$-edge-primitive graph. Let 
$e=\{v,w\}$ be an edge of $\Gamma$. Then $H=G_v$ is corefree in
$G$. Since $G$ is arc-transitive, there exists $g\in G$ such that
$v^g=w$ and $w^g=v$. Moreover, $\Gamma\cong \Cos(G,H,HgH).$ Now
$G_{vw}=H\cap H^g$ which is an index two subgroup of 
$G_e=\la H\cap H^g,g\ra$. Since $G$ is edge-primitive, $E=G_e$ is
maximal in $G$ and $A=H\cap E=G_{vw}$ has index 2 in $E$.
\end{proof}

We also have the following lemma.

\begin{lemma}
\label{lem:spread}
Let $\Gamma=\Cos(G,H,HgH)$. Then for any
subgroup $L<G$ such that $H\cap H^g<L<H$, the graph $\Cos(G,L,LgL)$ is a spread
of $\Gamma$.
\end{lemma}
\begin{proof}
Let $v$ be the vertex of $\Sigma=\Cos(G,L,LgL)$ corresponding to $L$
and $w$ the vertex adjacent to $v$ corresponding to $Lg$. Then $B=v^H$
is a block of imprimitivity for $G$ on $V\Sigma$ containing $v$ and
the corresponding block containing $w$ is $B^g$. Let
$\mathcal{B}=\{B^k\mid k\in G\}$. Since $\Gamma$ is
$G$-arc-transitive, so is $\Sigma_{\mathcal{B}}$ and $H$ is the
stabiliser of the vertex of $\Sigma_{\mathcal{B}}$ given by the block $B$.
Hence $\Sigma_{\mathcal{B}}= \Cos(G,H,HgH)=\Gamma$.  Now the
stabiliser of the block $B^g$ is $H^g$ and $H\cap H^g<L$. Let $(x,y)$
be an arc of $\Sigma$ with $x\in B$ and $y\in B^g$. Then there exists $h\in G$
mapping $v$ to $x$ and $w$ to $y$. Since $B$ is a block of
imprimitivity, $h\in H\cap H^g<L$ and so $h\in L\cap L^g$. Thus $h$
fixes $v$ and $w$ and so $\{v,w\}$ is the only edge between the blocks
$B$ and $B^g$. Hence $\Sigma$ is a spread of $\Gamma$.
\end{proof}

One easy way of constructing edge-primitive graphs is to look for
novelty maximal subgroups. Given a group $G$ with a normal subgroup $N$,
we say that a maximal subgroup $E$ of $G$ not containing $N$ is a
\emph{novelty} if $E\cap N$ is not maximal in $N$. Thus if
$N$ is an index two subgroup of $G$, every novelty maximal subgroup
$E$ of $G$ gives rise to a $G$-edge-primitive graph with edge
stabiliser $E$, arc stabiliser $A=E\cap N$ and vertex stabiliser $H$,
where $H$ is a proper subgroup of $N$ properly containing $A$. This
phenomenon lies behind Examples \ref{eg:pthyp} and \ref{eg:symp4}. We
also have the following example.

\begin{example}
\label{eg:M12}
{\rm Let $T$ be the Mathieu group $M_{12}$ and $G=\Aut(T)$. From the Atlas
\cite[p 33]{atlas}, $G$ has maximal subgroups $E\cong S_5$ and 
$H\cong \PGL(2,11)$ such that $A=E\cap H\cong A_5$ and 
$H\cap T=\PSL(2,11)$ is a maximal subgroup of $T$. The subgroup $E$ is a novelty maximal. Let 
$g\in E\backslash A$. Then by Proposition \ref{prn:general},
$\Gamma=\Cos(G,H,HgH)$ is $G$-edge-primitive. As $H$ is maximal in $G$
it follows that $G$ acts primitively on $V\Gamma$.
Note that $A\leqslant T$ and so $TA\neq G$. Hence $T$ acts transitively
on vertices and edges but not on arcs. Moreover, as $A$ is
selfnormalising in $T$, we have $A<H\cap T<T$ and $A$ is the
stabiliser in $T$ of an edge. Thus $\Gamma$ is 
$T$-edge-quasiprimitive, but not $T$-edge-primitive. Moreover,
$\Gamma$ is $G$-locally imprimitive and letting $B=H\cap T$, we see
that $\Gamma$ is the quotient graph of the
bipartite graph $\Sigma=\Cos(G,B,BgB)$. The graph $\Sigma$ is
$G$-edge-primitive and $(G,2)$-arc-transitive such that
$G^{E\Sigma}=G^{E\Gamma}$ and is 
$G$-vertex-biquasiprimitive, but not $G$-vertex-biprimitive. There is
a partition $\mathcal{P}$ of $V\Sigma$ into blocks of size two such
that $\Sigma_{\mathcal{P}}=\Gamma$. Each block of $\mathcal{P}$ has
one vertex in each bipartite half of $\Sigma$, and there is at most
one edge between any two blocks.
}%
\end{example}

We have the following general construction of locally imprimitive,
edge-primitive graphs.

\begin{construction}
{\rm Let $E$ be an almost simple primitive permutation group of degree $n$
such that $E$ has an index 2 subgroup $A$ which preserves a nontrivial
partition of the $n$ points into $l$ parts of size $k$. Let 
$H=S_k\Wr S_l$ and $G=S_n$. Suppose that $E$ is a maximal subgroup of
$G$ and let $g\in E\backslash A$. Then by Proposition
\ref{prn:general}, the graph $\Gamma=Cos(G,H,HgH)$ is
$G$-edge-primitive. If $A$ is not maximal in $H$ then $\Gamma$ is
$G$-locally imprimitive.

The requirements for $A$ and $E$ are often satisfied.  
An infinite family of examples is where $E=\Aut(\PSL(d,q))$ for 
$d\geq 3$ and $A=\PGammaL(d,q)$. Let $n=(q^d-1)(q^{d-1}-1)/(q-1)^2$,
the number of point-hyperplane incident pairs. Then by \cite{LPS87},
$E$ is maximal in $G=S_n$. However, $A$ is imprimitive and preserves a
partition of $l=(q^d-1)/(q-1)$ parts of size
$k=(q^{d-1}-1)/(q-1)$. Moreover, $A$ is not maximal in $H=S_l\Wr S_k$
since it is contained in $S_k\Wr \PGammaL(d,q)$. Thus  $\Gamma$ is
$G$-locally imprimitive.  
}%
\end{construction}

\section{Initial Analysis}

We begin by noting the following lemma.
\begin{lemma}
\label{lem:con}
If $\Gamma$ is a disconnected $G$-edge-primitive graph then either
$\Gamma$ is a union of isolated vertices and single edges, or $\Gamma$
is a union of isolated vertices and a connected $G$-edge-primitive graph.
\end{lemma}
\begin{proof}
Each connected component which contains an edge forms a block of
imprimitivity for $G$ on edges. Thus either each connected component
consists of zero or one edge, or there is a unique connected component with
at least one edge.
\end{proof}

Next we look at vertex-transitivity.
\begin{lemma}
\label{lem:qpfaithful}
Let $\Gamma$ be a connected $G$-edge-quasiprimitive graph. Then either $G$ is
vertex-transitive, or $\Gamma$ is bipartite and $G$ has two orbits on
vertices. Moreover, in the latter case, either $\Gamma$ is a
star or $G$ acts faithfully and quasiprimitively on each of its two
orbits.
\end{lemma}
\begin{proof}
Since $G$ is edge-transitive, either $G$ is vertex-transitive or
$\Gamma$ is bipartite and the two orbits $\Delta_1$, $\Delta_2$ of $G$ on $V\Gamma$
are the two parts of the bipartition. Suppose that we are in the
latter case and let $N$ be a nontrivial normal subgroup of $G$. Then
$N$ acts transitively on $E\Gamma$ and so, since $\Gamma$ is
connected, $N$ acts transitively on both $\Delta_1$ and
$\Delta_2$. Thus either $|\Delta_1|=1$ and $\Gamma$ is a star, or $G$
acts faithfully and quasiprimitively on each of its two
orbits.
\end{proof}

In the edge-primitive case things are more restricted.
\begin{lemma}
\label{lem:transitive}
Let $\Gamma$ be a connected $G$-edge-primitive graph. Then either
$\Gamma$ is a star or $G$ is vertex-transitive.  
\end{lemma}
\begin{proof}
Suppose that $G$ is vertex-intransitive. Then as $G$ is edge-transitive,
$\Gamma$ is a bipartite graph with the orbits of $G$ being the two
bipartite halves $\Delta_1$ and $\Delta_2$. Let $v \in \Delta_1$ and
$B=\{ \{v,w\}\mid w \in \Gamma(v)\}$.   Then $B$ forms a 
block of
imprimitivity for $G$ on edges. Thus either $|\Gamma(v)|=1$ or
$\Delta_1=\{v\}$. Since $\Gamma$ is connected, it follows that
$\Gamma$ is a star. 
\end{proof}

We can now show that all nontrivial edge-primitive graphs are
arc-transitive.

\begin{lemma}
\label{lem:arctrans}
Let $\Gamma$ be a connected $G$-edge-primitive graph. Then one of the
following holds:
\begin{enumerate}
\item $\Gamma$ is a star; 
\item $\Gamma$ is a cycle of prime length $p$, and $G$ is a cyclic group
 of order $p$;
\item $\Gamma$ is $G$-arc-transitive. 
\end{enumerate}
\end{lemma}
\begin{proof}
By  Lemma \ref{lem:transitive}, either case $(1)$ holds or $G$ is
vertex-transitive. Suppose now that $G$ is vertex-transitive but not
arc-transitive. Then for an edge $e=\{v,w\}$ we have 
$G_e=G_{vw}=G_v \cap G_{w}$. However, as $G$ acts primitively on
edges, $G_e$ is a maximal subgroup of $G$. Thus $G_v=G_{w}$ for every
pair of adjacent vertices.  But $\Gamma$ is connected, and so $G_v$
fixes every vertex of $\Gamma$. This implies that $G_v=1=G_e$ and so $G$
acts regularly on vertices and on edges. Thus $\Gamma$ has the
same number of edges as vertices and so the connectivity of $\Gamma$ implies that it is a cycle. Furthermore, as
$G$ is primitive on edges this cycle has a prime number of edges and
hence vertices. Moreover,  as $G$ is not arc-transitive, $G$ is
cyclic. Thus either case $(2)$ or $(3)$ holds. 
\end{proof}

Lemma \ref{lem:arctrans} does not hold for $G$-edge-quasiprimitive
graphs. In particular, the graph $\Gamma$ in Example \ref{eg:M12} is
$T$-edge-quasiprimitive, $T$-vertex-transitive but not
$T$-arc-transitive.

Next we look at the action of $G$ on vertices.

\begin{lemma}
\label{lem:qpbiqp}
Let $\Gamma$ be a connected $G$-vertex transitive, $G$-edge-quasiprimitive
graph. Then $G$ is either quasiprimitive or biquasiprimitive on the set
of vertices of $\Gamma$.
\end{lemma}
\begin{proof}
Let $N$ be a nontrivial normal subgroup of $G$. Then $N$ is transitive
on edges and so is either transitive on vertices or $\Gamma$ is
bipartite and $N$ has two orbits on the vertex set. Thus $G$ is either
quasiprimitive or biquasiprimitive on $V\Gamma$.
\end{proof}

In the edge-primitive case we can actually reduce to the situation
where $G$ is either primitive or biprimitive on vertices.
\begin{proof}{\bf (of Theorem \ref{thm:main})}
By Lemma \ref{lem:arctrans} $G$ is arc-transitive.
Suppose that $G$ is neither primitive nor
biprimitive on $V\Gamma$. Then there exists a $G$-invariant partition
$\mathcal{B}$ of $V\Gamma$ with at least three parts. Since $\Gamma$
is connected and edge-transitive, the edges of $\Gamma$ occur between
the parts of $\mathcal{B}$, that is, there are no edges within
parts. Let $\Gamma_{\mathcal{B}}$ be the quotient graph of $\Gamma$
with respect to the partition $\mathcal{B}$. Given
$B_1,B_2\in\mathcal{B}$ which are adjacent in $\Gamma_{\mathcal{B}}$,
the set of edges of $\Gamma$ between vertices of $B_1$ and vertices of
$B_2$ forms a block of imprimitivity for $G$. Hence there is a unique
edge in $\Gamma$ between vertices of $B_1$ and vertices of $B_2$. Thus
$\Gamma$ is a spread of $\Gamma_{\mathcal{B}}$ and
$G^{E\Gamma}\cong G^{E\Gamma_{\mathcal{B}}}$. Moreover, if $g\in G$ fixes
each part of $\mathcal{B}$, then $g$ fixes each edge of $\Gamma$. Thus
$G$ acts faithfully on $\mathcal{B}$.  Moreover, by choosing
$\mathcal{B}$ to be a maximal $G$-invariant partition with at least
three parts, $G$ is either primitive or biprimitive on the set of
vertices of $V\Gamma_{\mathcal{B}}$. Let $v\in B_1$ and $w\in B_2$ be
the unique pair of adjacent vertices in $B_1\cup B_2$. Then
$G_{B_1B_2}=G_{vw}<G_v<G_{B_1}$, since $G$ is arc-transitive and
$|B_1|>1$. Hence $\Gamma_{\mathcal{B}}$ is $G$-locally imprimitive. 

Conversely, suppose that $\Sigma$ is a $G$-edge-primitive, $G$-locally
imprimitive graph. Let $\{\alpha,\beta\}\in E\Sigma$. Then there exists
a subgroup $H$ such that $G_{\alpha\beta}<H<G_{\alpha}$. Since $G$ is
arc-transitive, there exists $g\in G$ such that $g$ interchanges
$\alpha$ and $\beta$.  Thus $H\cap H^g\leqslant G_{\alpha}\cap
G_{\beta}$, but since $g$ normalises $G_{\alpha\beta}$ we have 
$H\cap H^g=G_{\alpha\beta}$. Moreover, $g^2\in G_{\alpha\beta}\leqslant H$. 
Thus we can define the graph
$\Gamma=Cos(G,H,HgH)$. Let $v$ be the vertex of $\Gamma$ given by $H$
and $w$ the vertex given by the coset $Hg$. Then $e=\{v,w\}$ is an
edge and $G_e=\la H^g\cap H,g\ra=\la
G_{\alpha\beta},g\ra=G_{\{\alpha,\beta\}}$. Hence 
$G^{E\Gamma}\cong G^{E\Sigma}$ and so $\Gamma$ is
$G$-edge-primitive. Since $H=G_v<G_{\alpha}<G$, it follows that
$B_1=v^{G_{\alpha}}$ is a block of imprimitivity for $G$ on $V\Gamma$. Let
$\mathcal{B}$ be the corresponding system of imprimitivity. Now
$v^{G_{\alpha}g}=v^{gg^{-1}G_{\alpha}g}=w^{G_{\beta}}$ and so
$B_2=w^{G_{\beta}}$ is the block of $\mathcal{B}$ containing $w$.
Moreover, $(v,w)$ is the
unique edge between the two blocks $B_1$ and $B_2$.  Then as
$G_{\alpha}=G_{B_1}$ and $g$ interchanges the edge $\{B_1,B_2\}$ of the
quotient graph $\Gamma_{\mathcal{B}}$ we have that $\Gamma_{\mathcal{B}}\cong
Cos(G,G_{\alpha},G_{\alpha}gG_{\alpha})\cong \Sigma$. 
\end{proof}

We also have the following lemma in the vertex-biquasiprimitive case.
\begin{lemma}
\label{lem:faithful}
Let $\Gamma$ be a $G$-vertex-biquasiprimitive graph which is not
complete bipartite. Then $G^+$ is faithful on each orbit.
\end{lemma}
\begin{proof}
Let $\Delta_1$ and $\Delta_2$ be the two orbits of $G^+$ on vertices
and suppose that $G^+$ is unfaithful on $\Delta_1$.  Let $K_1$ be the
kernel of the action of $G^+$ on $\Delta_1$ and $K_2$ be the kernel of
the action of $G^+$ on $\Delta_2$. Then as $G$ is
vertex-transitive, there exists $g\in G$ such that
$K_1^g=K_2$. Moreover, $1\neq K_1\times K_2\norml G$. Since $G$ is
vertex-biquasiprimitive, it follows that $K_1$ is transitive on
$\Delta_2$ and $K_2$ is transitive on $\Delta_1$. Since $K_1$ fixes
each vertex in $\Delta_1$, we have that each vertex of $\Delta_1$ is
adjacent to each vertex of $\Delta_2$. Thus $\Gamma$ is complete
bipartite.
\end{proof}

We can determine all $n$ and $G$ such that $K_{n,n}$ is
$G$-edge-primitive and $G^+$ acts faithfully on each bipartite half. 
\begin{theorem}
\label{thm:compbipartite}
Let $\Gamma=K_{n,n}$ be a $G$-edge-primitive graph. Then one of the
following holds: 
\begin{enumerate}
\item $n=6^k$ and $\soc(G^+)=A_6^k$.
\item $n=12^k$ and $\soc(G^+)=M_{12}^k$.
\item $n=(q^2(q^2-1)/2)^k$ and $\soc(G^+)=\PSp(4,q)$ with $q$ even.
\item There exists a primitive group $H$ of degree $n$ with a transitive
  but not regular
  normal subgroup $K$ and automorphism $\phi$ such that 
  $G^+=\{(hk_1,h^{\phi}k_2)\mid k_1,k_2\in K, h\in H\}$, and 
  $(g,1_H)(1,2)\in G$ for some $g\in H$ interchanges the two $G^+$ orbits where $\phi^2$
  is conjugation by $g$. 
\end{enumerate}
\end{theorem}
\begin{proof}
Let $\Delta_1$ and $\Delta_2$ be the two bipartite halves of $\Gamma$. 
Suppose that $G^+$ is imprimitive on $\Delta_1$ and let
$\mathcal{P}_1$ be a system of imprimitivity for $G^+$ on $\Delta_1$.  Then there exists a
system of imprimitivity $\mathcal{P}_2$ of $G^+$ on $\Delta_2$ such
that $\mathcal{P}_2=\mathcal{P}_1^g$ for all $g\in G\backslash G^+$.
Let $B_1\in \mathcal{P}_1$ and $B_2\in\mathcal{P}_2$. Then
$C=\{(v,w)\mid v\in B_1,w\in B_2\}$ is a block of imprimitivity for
$G$ on $E\Gamma$. Hence $G^+$ is primitive on each bipartite half.  

Let $v\in \Delta_1$ and $w\in \Delta_2$. By Lemma \ref{lem:arctrans},
$G$ is arc-transitive. Thus $G_v$ is transitive on $\Delta_2$ and so
$G^+=G_vG_w$. Suppose first that $G^+$ is faithful on $\Delta_1$ and
$\Delta_2$. Since $G_w=G_v^g$ for some $g\in G$ with $g^2\in G^+$,
it follows that $G^+, G_v$ and $G_w$ are determined by 
\cite[Theorem 1.1]{baumeister}. Either
$G=\AGL(3,2)\Wr K$ for some transitive subgroup $K$ of $S_k$, or
$\soc(G^+)=T^k$ where $T$ is one of $\POmega^+(8,q)$, $\PSp(4,q)$ $q>2$
even, $A_6$ or $M_{12}$. 

If $G=\AGL(3,2)\Wr K$, then 
$G_{\{v,w\}}=\la (C_7\rtimes C_3) \Wr K, (\alpha,\ldots,\alpha)\ra$
where $\alpha$ is an automorphism of $\AGL(3,2)$ interchanging the two
conjugacy classes of complements of $C_2^3$. Hence 
$G_{\{v,w\}}<C_2^{3k}\rtimes G_{\{v,w\}}<G$ and so $G$ is not
edge-primitive. 

If $N=\soc(G^+)=\POmega^+(8,q)^k$ then $G^+\leqslant H^k$ where $H$ is
an extension of $\POmega^+(8,q)$ by field automorphisms, 
$N_v=\POmega(7,q)$, $N_{vw}=G_2(q)$ and $n=q^4(q^4-1)$ 
\cite[Theorem 1.1]{baumeister}.  Since
$N_v^g=N_w$ for some $g\in G\backslash G^+$ such that $g^2\in G^+$, it
follows that $g$ does not induce a triality automorphism of
$\POmega^+(8,q)$. Hence by \cite{O8+}, $G_{\{vw\}}$ is not maximal in
$G$, and so $G$ is not edge-primitive. Thus $\soc(G^+)$ and $n$ are as
listed in the statement of the theorem.

Suppose next that $G^+$ is unfaithful on $\Delta_1$ and $\Delta_2$.
Let $K_1$ be the kernel of the action of $G^+$ on $\Delta_1$ and $K_2$
be the kernel of the action of $G^+$ on $\Delta_2$. Then 
$K_1\times K_2\norml G$ and so is transitive on $E\Gamma$. Hence $K_1$
acts faithfully and transitively on $\Delta_2$ and $K_2$
acts transitively and faithfully on $\Delta_1$. Let
$H=(G^+)^{\Delta_1}$ and $K=(K_2)^{\Delta_1}$. Then $H$ is a primitive
permutation group with transitive normal subgroup $K$.  Now
$G\leqslant H\Wr S_2$ and $G=\la G^+, (g,1_H)(1,2)\ra$ for some 
$g\in H$. Then $K_2=\{(1_H,k)\mid k\in K\}$ and
$K_1=K_2^{(g,1_H)(1,2)}=\{(k,1_H)\mid k\in K\}$. Furthermore, there exists an
automorphism $\phi$ of $H$ such that 
$G^+=\{(hk_1,h^{\phi}k_2)\mid h\in H,k_1,k_2\in K\}$. Since
$(g,1_H)(1,2)$ normalises $G^+$ it follows that $\phi^2$ is conjugation
by $g$. If $K$ is regular then $H=K\rtimes H_v$ and so $G^+=
\la K\times K\ra \rtimes \{(h,h^{\phi})\mid h\in H_v\}$. Moreover,
$G_e=\la \{(h,h^{\phi})\mid h\in H_v\},(g,1_H)(1,2)\ra <\la
\{(h,h^{\phi})\mid h\in H\},(g,1_H)(1,2)\ra<G$, contradicting $G_e$ being
  maximal in $G$. Thus $K$ is not regular.
\end{proof}

\section{Primitive and quasiprimitive types}
\label{sec:prelim}

In this section we describe the subdivision of primitive and
quasiprimitive groups into 8 types given in \cite{praegerbcc}.  This
description is in terms of the action of the minimal normal
subgroups. If $N$ is a minimal normal subgroup of a group $G$ then
$N\cong T^k$ for some finite simple group $T$. Moreover, if $G$ is
quasiprimitive then $G$ has at most two minimal normal subgroups.

\vspace{0.5cm}
{\bf HA:} A quasiprimitive group is of type HA if it has a unique
minimal normal subgroup $N$ and $N$ is elementary abelian. If $|N|=p^d$
for some prime $p$, then $G$ can be embedded in $\AGL(d,p)$ in its
usual action on a $d$-dimensional vector space over $\GF(p)$ with $N$
identified as the group of all translations.

\vspace{0.5cm}
{\bf HS and HC:} These two classes consist of all quasiprimitive
groups with two minimal normal subgroups. In both cases, the two
minimal normal subgroups are regular and nonabelian. For type HS, the
two minimal normal subgroups are simple, while for type HC the two
minimal normal subgroups are isomorphic to $T^k$ for some $k\geq 2$
and $T$ nonabelian simple.

All quasiprimitive groups of type HA, HS and HC are in
fact primitive. For the remaining five types the groups may or may not
be primitive.

\vspace{0.5cm}
{\bf AS:} This class consists of all groups $G$ such that
$T\leqslant G\leqslant\Aut(T)$ for some finite nonabelian simple
group, that is, $G$ is an almost simple group. Note that any action of
an almost simple group with $T$ transitive is quasiprimitive.

\vspace{0.5cm}
{\bf TW:} This type consists of all quasiprimitive
groups $G$ with a unique minimal normal subgroup $N\cong T^k$, for some finite
nonabelian simple group $T$ and positive integer $k\geq 2$, such that
$N$ is regular. Thus $G=N\rtimes G_{\omega}$ and can be constructed as
a twisted wreath product (see \cite{baddeleyTW}). If $G$ is primitive
then $G_{\omega}$ normalises no nontrivial proper subgroup of $N$. The
following lemma gives us a necessary and sufficient condition for a
quasiprimitive TW group to be primitive. 
\begin{lemma}\cite[Lemmas 3.1 and 3.2]{baddeleyTW}
\label{lem:TWmax}
Let $N\cong T^k$ for some finite nonabelian simple group $T$ and
$G=N\rtimes P$. Let $Q$ be the normaliser in $P$ of a simple direct
factor of $N$ and $\varphi:Q\rightarrow \Aut(T)$ be the homomorphism
induced by the action of $Q$ on this factor. Then $P$ is maximal in
$G$ if and only if $\Inn(T)\leqslant \varphi(Q)$ and there is no subgroup $H$ of $P$ with a homomorphism $\hat{\varphi}$ from $H$ to $\Aut(T)$ which extends $\varphi$.
\end{lemma}

\vspace{0.5cm}
Before describing the remaining three types of
quasiprimitive groups we need some definitions.
Let $N=T_1\times\cdots\times T_k$ for some nontrivial groups
$T_1,\ldots, T_k$. For each $i=1,\ldots,k$, let 
$\pi_i:N\rightarrow T_i$ be the natural projection map. Given a
subgroup $K$ of $N$, we say that $K$ is a \emph{subdirect product} of $N$ if
$\pi_i(K)=T_i$ for each $i=1,\ldots, k$, while we say that $K$ is a
\emph{diagonal subgroup} of $N$ if $K$ is isomorphic to each of its
projections, that is, $K\cong\pi_i(K)$ for all $i=1,\ldots,k$. 
If $T_1=T_2=\cdots=T_k$ and $\pi_i(g)=\pi_j(g)$ for all $g\in K$, we call
$K$ a \emph{straight diagonal subgroup}. A 
\emph{full diagonal subgroup} of $N$ is a subgroup which is both a
subdirect product and a diagonal subgroup.  

We call $K$ a \emph{strip} of $N$ if there exists some subset $J$ of
$\{1,\ldots,k\}$ such that $\pi_i(K)\cong K$ for all $i\in J$ while
$\pi_i(K)=1$ for all $i\notin J$. We refer to $J$ as the \emph{support} of
$K$. Note that a strip is a diagonal subgroup of 
$\prod_{i\in J}T_i$. We call $K$ a \emph{full strip} if it is a full
diagonal subgroup of $\prod_{i\in J}T_i$, while we say that it is 
\emph{nontrivial} if $|J|>1$. We say that two strips are
\emph{disjoint} if their supports are disjoint. Note that disjoint
strips commute.   

If $N=T_1\times\cdots\times T_k$, where the $T_i$ are pairwise
isomorphic nonabelian simple groups, a well known lemma (see for
example \cite{Scott}) says that if
$K$ is a subdirect product of $N$ then $K$ is the direct product of
pairwise disjoint full strips. The set of supports of these strips is
a partition $\mathcal{P}$ of $\{1,\ldots,k\}$. Note that if 
$K$ is normalised by a group $G$, then $G$ preserves $\mathcal{P}$ and if $G$ acts
transitively by conjugation on the set $\{T_1,\ldots,T_k\}$, then $G$
acts transitively on $\mathcal{P}$ and so the parts of $\mathcal{P}$
all have the same size.

\vspace{0.5cm}
{\bf SD:}
A quasiprimitive group $G$ acting on a set $\Omega$ is of type
  SD if
$G$ has a unique minimal normal subgroup $N$, $N\cong T^k$ for some
nonabelian simple group $T$, $k\geq 2$ and given $\omega\in\Omega$,
the point stabiliser $N_\omega$ is a full diagonal subgroup of
$N$. Conjugating $G$, if necessary, by an element of $\Sym(\Omega)$ we
may assume that $N_\omega=\{(t,\ldots,t)\mid t\in T\}$ and 
$G_\omega\leqslant\{(t,\ldots,t)\mid t\in\Aut(T)\}\rtimes S_k$. Since $N$
is a minimal normal subgroup of $G$ and $G=NG_\omega$, it follows that
$G_\omega$ acts transitively by conjugation on the set of $k$ simple
direct factors of $N$. A quasiprimitive group $G$ of type SD is
primitive, if and only if $G$ acts primitively on the set of $k$
simple direct factors of $N$.

\vspace{0.5cm}
{\bf CD:}
A quasiprimitive group $G$ acting on a set $\Omega$ is of type CD if
$G$ has a unique minimal normal subgroup $N$, $N\cong T^k$ for some
nonabelian simple group $T$, $k\geq 2$ and given $\omega\in\Omega$,
$N_\omega$ is a product of $\ell\geq 2$ full strips of $N$, that is,
$N_{\omega}\cong T^\ell$ and is a subdirect product of $N$. Note that
$G$ acts transitively by conjugation on the set of $k$ simple direct
factors of $N$ and preserves a partition $\mathcal{P}$ of
$\{1,\ldots,k\}$ given by the set of supports of the full strips.  
The group $G$ is a subgroup of $H\Wr S_\ell$ acting on
$\Omega=\Delta^\ell$, for some quasiprimitive group $H$ of type SD on
$\Delta$ with unique minimal normal subgroup $T^{k/\ell}$. In fact,
given $P\in \mathcal{P}$, the group $G_P$ induces $H$ on $\Delta$.
Moreover, $G$ is primitive if and only if $H$ is primitive and so $G$
is primitive if and only if for $P\in \mathcal{P}$, $G_P$ acts
primitively on $\mathcal{P}$.

Given two partitions 
$\mathcal{P}_1,\mathcal{P}_2$ of a set $\Omega$, we say that
$\mathcal{P}_1$ \emph{refines} $\mathcal{P}_2$ if each 
$P\in \mathcal{P}_2$ is a union of elements of $\mathcal{P}_1$.  This
defines a partial order on the set of all partitions of $\Omega$ and
we can define $\mathcal{P}_1\vee\mathcal{P}_2$ to be the smallest
partition of 
$\Omega$ refined by both $\mathcal{P}_1$ and $\mathcal{P}_2$. The
following lemma will be very handy in our analysis of SD and CD groups.
\begin{lemma}
\label{lem:intersectstrips}
Let $N=T_1\times\cdots\times T_k$ for some nontrivial groups $T_i$
and let $K_1,K_2$ be subgroups of $N$. For each $i=1,2$, suppose that
$K_i$ is a product of strips such that the set of supports of these
strips is the partition $\mathcal{P}_i$ of $\{1,\ldots,k\}$. Then
$K_1\cap K_2$ is a product of strips such that the set of supports of
these strips is $\mathcal{P}_1\vee \mathcal{P}_2$.
\end{lemma}
\begin{proof}
For each $P\in \mathcal{P}_1\vee \mathcal{P}_2$, let 
$$K_P=\{g\in K_1\cap K_2\mid \pi_i(g)=1 \text{ for all } i\notin P\}.$$
Then $X=\prod_{P\in\mathcal{P}_1\vee \mathcal{P}_2} K_P$ is a subgroup
of $K_1\cap K_2$. 

Let $g\in K_1\cap K_2$ such that $g\neq 1$, and let
$J$ be the set of all $i\in \{1,\ldots,k\}$ such that
$\pi_i(g)\neq 1$.  Since $g\in K_1$ it follows that $J$ is a union
of parts of $\mathcal{P}_1$ and since $g\in K_2$ it follows that $J$ is a union
of parts of $\mathcal{P}_2$. Hence $J$ is a union of $\ell$ parts of
$\mathcal{P}_1\vee \mathcal{P}_2$ for some $\ell\geq 1$. Thus each $K_P$
is a strip.  If $\ell=1$ then $g\in K_P$ for some $P$ and so $g\in X$. If
$\ell> 1$, let $P$ be one of the parts contained in $J$. Since 
$g\in K_1$ and $K_1$ is a product of strips, there exists $k_1\in K_1$ such
that $\pi_i(k_1)=\pi_i(g)$ for all $i\in P$ while $\pi_i(k_1)=1$ for all
$i\notin P$. Similarly, there exists $k_2\in K_2$ such that
$\pi_i(k_2)=\pi_i(g)$ for all $i\in P$ while $\pi_i(k_2)=1$ for all
$i\notin P$. Hence $k_1=k_2\in K_P\leqslant X$. Moreover, 
$gk_1^{-1}\in K_1\cap K_2$ and has support $J\backslash P$, a union of
$\ell-1$ parts of $\mathcal{P}_1\vee\mathcal{P}_2$.  It follows that
$g\in X$ and so $K_1\cap K_2$ is a product of the strips $K_P$,
whose supports are the parts of 
$\mathcal{P}_1\vee\mathcal{P}_2$.
\end{proof}

\vspace{0.5cm}
{\bf PA:}
A quasiprimitive group $G$ acting on a set $\Omega$ is of type PA if
$G$ has a unique minimal normal subgroup $N$, $N\cong T^k$ for some
nonabelian simple group $T$, 
$k\geq 2$ and given $\omega\in\Omega$, $N_\omega$ is a subdirect
product of $R^k$ for some $R<T$.  The following two lemmas will be
useful for determining primitivity. See for example, 
\cite[Lemma 2.7A]{dixon} for a proof of the first.

\begin{lemma}
\label{lem:PAprim1}
Let $B$ be a group with subgroup $H\neq 1$. Then for each
positive integer $k$, $H\Wr S_k$ is maximal in $B\Wr S_k$ if and
only if $H$ is maximal in $B$.
\end{lemma}

\begin{lemma}
\label{lem:PAprim2}
Let $T$ be a nonabelian simple group and let 
$T\leqslant A\leqslant B\leqslant \Aut(T)$. Suppose that $H$ is a
maximal subgroup of $B$ such that $B=TH$ and $T\cap H\neq 1$. Let
$$G=\la A^k, (b,\ldots,b)\mid b\in B\ra \rtimes S_k$$
and
$$L=\la (A\cap H)^k, (h,\ldots,h)\mid h\in H\ra\rtimes S_k.$$
Then $L$ is a maximal subgroup of $G$.
\end{lemma}
\begin{proof}
Let $M$ be a subgroup of $G$ containing $L$ and let 
$X=M\cap B^k$. Since $S_k\leqslant M$ it follows that 
$\pi_i(X)\cong \pi_j(X)$ for all $i$ and $j$. Since $L\leqslant M$ we
have $H\leqslant \pi_i(X)$ and since $H$ is maximal in $B$ it follows
that $\pi_i(X)=B$ for all $i$. Hence $X\cap T^k$ is a subdirect
product of $T^k$. However, $H\cap T\neq 1$ and 
$(H\cap T)^k\leqslant X$. Thus $X\cap T^k=T^k$. Since $B=TH$ we also
have $A=T(A\cap H)$. Then as $(A\cap H)^k\leqslant X$ it follows that
$A^k\leqslant X$. Thus $X=G\cap B^k$ and so $M=G$, that is, $L$ is maximal. 
\end{proof}

\section{Constructions}
\label{sec:con}

All the examples in Section \ref{sec:eg} had $G$ an almost simple
group. In this section we provide some general constructions for
$G$-edge-quasiprimitive graphs where $G$ is not of type AS.

Our first construction takes a $B$-edge-primitive graph where $B$ is
an almost simple group such that $B\neq \soc(B)$, and builds a
$G$-edge-primitive graph where $G$ is 
primitive of type PA on edges and primitive of type PA on vertices.

\begin{construction}\label{con:PAPA}
(Primitive PA on vertices and primitive PA on edges)
{\rm %
Let $\Sigma$ be a $B$-edge-primitive, $B$-vertex-primitive graph such
that $B$ is an almost simple group with socle $T<B$. Note that Example
\ref{eg:M12} is such a graph. Then there exist a maximal subgroup $H$
of $B$ and $g\in B\backslash H$, such that $g^2\in H$ and  
$\Sigma\cong \Cos(B,H,HgH)$. 
Let  
$$G=\la T^k,(b,\ldots,b)\mid b\in B\ra \rtimes S_k$$
and 
$$L=\la (T\cap H)^k,(h,\ldots,h)\mid h\in H\ra \rtimes S_k.$$
Then letting $\sigma=(g,\ldots,g)$ we define the coset graph
$\Gamma=\Cos(G,L,L\sigma L)$. 
}%
\end{construction}
\begin{lemma}
The graph $\Gamma=\Cos(G,L,L\sigma L)$ given by Construction
\ref{con:PAPA} is $G$-edge-primitive of type PA and
$G$-vertex-primitive of type PA. Moreover, if
$\Sigma$ given in Construction \ref{con:PAPA} is $B$-locally imprimitive then $\Gamma$ is $G$-locally
imprimitive.
\end{lemma}
\begin{proof}
Since $T<B$ and $B$ is primitive, it follows that $H\neq 1$. Then as
$H$ is a maximal subgroup of $B$, Lemma \ref{lem:PAprim2} implies that
the action of $G$ on 
$V\Gamma=[G:L]$ is primitive of type PA. Let $v$ be the vertex given
by the coset $L$ and $w$ be the 
adjacent vertex given by $L\sigma$. Then $G_w=L^{\sigma}$ and  
$$G_v\cap G_w=
\la(T\cap H\cap H^g)^k,(h,\ldots,h)\mid h\in H\cap H^g\ra\rtimes S_k.$$
Furthermore, 
\begin{align*}
G_{\{v,w\}} &=\la G_v\cap G_w,\sigma\ra \\
  &= \la(T\cap H\cap H^g)^k,(h,\ldots,h)\mid h\in \la H\cap
     H^g,g\ra\ra\rtimes S_k
\end{align*}
which  by Lemma \ref{lem:PAprim2}, is a maximal subgroup of $G$ since
$\la H\cap H^g,g\ra$ is a maximal subgroup of $B$. Hence $G$ acts
primitively on $E\Gamma$ of type PA.

If $\Sigma$ is $B$-locally imprimitive there exists a subgroup $R$
such that $H\cap H^g<R<H$. It follows that $G_{vw}$ is not maximal in
$L$ and so $\Gamma$ is $G$-locally imprimitive.
\end{proof}

We have the following construction which takes a $B$-edge-primitive
bipartite graph such that $B$ is almost simple and $B^+$ is primitive
on each bipartite half, and builds a $G$-edge-primitive bipartite
graph with $G$ primitive of type PA on edges and 
$G^+$ primitive of type PA on each of the bipartite halves.

\begin{construction}
\label{con:PAbiqp}
(Primitive {\rm PA} on edges and biprimitive on vertices with
$G^+$ primitive of type {\rm PA})
{\rm %
Let $\Sigma$ be a bipartite connected $B$-edge-primitive graph such
that $B$ is an almost simple group with socle $T$ such that $B^+$ acts
primitively on each bipartite half. Then there exist a corefree maximal
subgroup $H$ of $B^+$ and $g\in B\backslash B^+$ such that $g^2\in H$
and $\Sigma=\Cos(B,H,HgH)$. Let $\sigma=(g,\ldots,g)$,
$$G=\la (B^+)^k,\sigma\ra \rtimes S_k,$$
and $L=H^k\rtimes S_k$. Define $\Gamma=\Cos(G,L,L\sigma L)$. 
} %
\end{construction}
\begin{lemma}
The connected bipartite graph $\Gamma=\Cos(G,L,L\sigma L)$ yielded by
Construction \ref{con:PAbiqp} is $G$-edge-primitive of type PA and 
$G$-biprimitive on vertices such that $G^+$ acts primitively of
type PA on both of its vertex orbits. Moreover, $\Gamma$ is
$G$-locally primitive if and only if $\Sigma$ is $B$-locally primitive. 
\end{lemma}
\begin{proof}
Since $\Sigma$ is connected we have $\la H,g\ra=B$. It follows that 
$\la L,\sigma\ra=G$ and so $\Gamma$ is connected. The stabiliser in $G$
of the edge $e=\{L,L\sigma\}$ is 
$\la (H\cap H^g)^k,\sigma \ra \rtimes S_k$ which by Lemma
\ref{lem:PAprim2} is a maximal subgroup of $G$ since 
$\la H\cap H^g,g\ra$, the stabiliser in $B$ of an edge in $\Sigma$, is
a maximal subgroup of $B$. Hence $G$ acts primitively of type PA on
$E\Gamma$. The index two subgroup $G^+=B^+ \Wr S_k$ of $G$ has two
orbits on $V\Gamma$. Hence $\Gamma$ is bipartite. Moreover, since $H$
is a maximal subgroup of $B^+$ it follows from Lemma \ref{lem:PAprim1}
that $G^+$ acts primitively of type PA on each of the bipartite
halves.    

Since $\la H\cap H^g,g\ra$ is maximal in the almost simple group
$B^+$, we have $H\cap H^g\neq 1$. Thus by Lemma \ref{lem:PAprim1}
$(H\cap H^g)^k\rtimes S_k$ is maximal 
in $H^k\rtimes S_k$ if and only if $H\cap H^g$ is maximal in $H$, and 
$\Gamma$ is $G$-locally primitive if and only if $\Sigma$ is
$B$-locally primitive.   
\end{proof}

\begin{remark}
\label{rem:notqp}
{\rm 
Suppose that in Construction \ref{con:PAbiqp}, we let $k=2$ and let
$\overline{G}=(B^+)^2\rtimes \la (g,g)(1,2)\ra\leqslant G$. Then
$\overline{G}_e=(H\cap H^g)^2\rtimes \la (g,g)(1,2)\ra$, which is a
maximal subgroup of $\overline{G}$.  Thus $\overline{G}$ is
edge-primitive of type PA and biquasiprimitive on vertices. Moreover,
$(\overline{G})^+=(B^+)^2$ and $\overline{G}_v=H^2$. Hence
$(\overline{G})^+$ is not quasiprimitive on each bipartite half of
$\Gamma$.
}%
\end{remark}

We now give a general construction of $G$-edge-quasiprimitive graphs 
for which the action of $G$ on edges is of type SD or CD and $G$ is
vertex-transitive.
\begin{construction}
\label{con:SDCDtrans}
(Quasiprimitive {\rm SD} or {\rm CD} on edges and vertex-transitive)
{\rm %
Let $G$ be a quasiprimitive group on a set $\Omega$ of type SD or CD
with socle $N=T^k$. Let $\omega\in\Omega$ and let $\mathcal{P}$ be the $G$-invariant partition of $\{1,\ldots,k\}$ given by the set of supports of the full strips of $N_\omega$. If $G$ is of type SD then $\mathcal{P}=\{\{1,\ldots,k\}\}$ while if $G$ is of type CD then
$\mathcal{P}$ is a nontrivial system of imprimitivity for $G$. 
Suppose that $G$ has an index two subgroup $G^+$ which leaves
invariant two distinct partitions $\mathcal{P}_1$ and $\mathcal{P}_2$ of
$\{1,\ldots,k\}$ which are interchanged by $G$, and such
that $\mathcal{P}_1\vee\mathcal{P}_2=\mathcal{P}$. 

Let $L=G_{\omega}$. Conjugating by a suitable
element of $\Sym(\Omega)$ we may assume that each $h\in L$ is of the
form $(t_1,\ldots,t_k)\sigma$ where 
$t_i\in\Aut(T)$, $\sigma\in S_k$,  $\sigma$ preserves $\mathcal{P}$,
and if $i,j$ belong to the same part of $\mathcal{P}$ then $t_i=t_j$. 
Since $L^{\{1,\ldots,k\}}=G^{\{1,\ldots,k\}}$, it follows that $L$ has
an index two subgroup $L^+$ which leaves $\mathcal{P}_1$ and
$\mathcal{P}_2$ invariant. Moreover, $L=\la L^+,g\ra$ for some element 
$g=(t_1,\ldots,t_k)\sigma\in G$, where $\sigma$ interchanges
$\mathcal{P}_1$ and $\mathcal{P}_2$. For a subset $I$ of
$\{1,\ldots,k\}$, let $T_I$ be the straight full strip of $N$ whose
support is $I$. Let $N_1=\prod_{I\in\mathcal{P}_1} T_I$ and let
$H=N_G(N_1)$. Then as $L^+$ leaves $\mathcal{P}_1$ invariant and $G^+$
is the stabiliser of $\mathcal{P}_1$ in $G$,  we have 
$L^+\leqslant H\leqslant G^+$. Moreover, since $G^+=NL^+$, if $nl\in H$ with
$n\in N$ and $l\in L$ then $n\in N_N(N_1)=N_1$. Thus
$H=N_1L^+$. Furthermore, $H^g=N_2L^+$ where
$N_2=\prod_{I\in\mathcal{P}_2}T_I$. Since $g^2\in L^+$ it follows that
$g^2 \in H$ and we can define $\Gamma=\Cos(G,H,HgH)$.
}%
\end{construction}

\begin{lemma}
\label{lem:tranSDCDcon}
The graph $\Gamma=\Cos(G,H,HgH)$ obtained from Construction
\ref{con:SDCDtrans} is $G$-edge-quasiprimitive of type SD or CD such
that $G$ is vertex-biquasiprimitive. Moreover, $\Gamma$ is $G$-locally
primitive if and only if $\mathcal{P}_1$ is the coarsest
$L^+$-invariant partition of $\mathcal{P}_1\vee\mathcal{P}_2$
refined by $\mathcal{P}_1$.
\end{lemma}
\begin{proof}
Let  $v$ be the vertex corresponding to the
coset $H$ and $w$ be the vertex corresponding to $Hg$. Then
$e=\{v,w\}$ is an edge and $G_{vw}=H\cap H^g=(N_1\cap N_2)L^{+}$.
Elements of $N_1\cap N_2$ are constant on the parts of
$\mathcal{P}_1$ and the parts of $\mathcal{P}_2$, hence are constant
on the parts of $\mathcal{P}_1\vee\mathcal{P}_2=\mathcal{P}$. Thus 
$N_1\cap N_2=L\cap N$. Hence $G_{vw}=L^+$ and $G_e=L$. It follows that
$G^{\Omega}\cong G^{E\Gamma}$ and so $\Gamma$ is
$G$-edge-quasiprimitive with $G^{E\Gamma}$ of type SD or CD. Moreover,
$G^+$ has two orbits on $V\Gamma$ and so $\Gamma$ is bipartite. Since
$N$ is the unique minimal normal subgroup of $G$ and has two vertex
orbits it follows that $G$ is vertex-biquasiprimitive. Further,
$G_{vw}=L^+$
is maximal in $G_v=H$ if and only if $\mathcal{P}_1$ is the  coarsest
$L^+$-invariant partition of $\mathcal{P}_1\vee\mathcal{P}_2$ refined
by $\mathcal{P}_1$. Hence the statement regarding local
primitivity follows.
\end{proof}

We now demonstrate the various vertex actions which can be yielded by Construction \ref{con:SDCDtrans}.
\begin{example}
\label{eg:SD}
{\rm A suitable choice for $G$ primitive of type SD in Construction
\ref{con:SDCDtrans}, is $N\rtimes K$, where $K=S_d\Wr S_2$ for some
$d\geq 3$, $N=T^{d^2}$ and $G^+=N\rtimes S_d^2$. Here  
$\mathcal{P}_1$ corresponds to the set of orbits of $1\times S_d$ on
the $d^2$ simple direct factors of $N$ (that is, the ``horizontal''
blocks) while $\mathcal{P}_2$ corresponds to set of orbits of
$S_d\times 1$ (that is, the ``vertical blocks'').  Note that $G^+$ is
primitive of type CD on each of its vertex orbits and $G$ is
vertex-biprimitive.
}%
\end{example}
\begin{example}
\label{eg:SDHC}
{\rm Let $G=T^4\rtimes \la (1,3,2,4)\ra$. Here
$\mathcal{P}_1=\{\{1,4\},\{2,3\}\}$,  
$\mathcal{P}_2=\{\{1,3\},\{2,4\}\}$, and
$\mathcal{P}=\{\{1,2,3,4\}\}$. Then $G$ is quasiprimitive but not
primitive of type SD on edges and $G^+=T^4\rtimes\la (1,2)(3,4)\ra$ is
primitive of type HC on each vertex orbit.
}%
\end{example}
\begin{example}
\label{eg:CDCD}
{\rm A suitable choice
for  $G$ primitive of type CD is $G=N\rtimes K$ where $N=T^{d^2m}$
and $K=(S_d^2)^m.2\rtimes S_m$ such that $K$ preserves the partition 
$\mathcal{P}$ of $m$ blocks of size $d^2$ with $d\geq 3$. Here $K$ has
an index two 
subgroup $K_1=S_d^2\Wr S_m$ with two systems of imprimitivity
$\mathcal{P}_1$ and $\mathcal{P}_2$ with $dm$ parts of size $d$,
interchanged by $K$. The partition $\mathcal{P}_1$ is the set of
orbits of $(1\times S_d)^m$ on the set of $d^2m$ simple direct factors
of $N$ (the set of horizontal blocks in each part of $\mathcal{P}$)
while $\mathcal{P}_2$ is the set of orbits of $(S_d\times 1)^m$ (the
set of vertical blocks of each part of $\mathcal{P}$). Moreover,  
$\mathcal{P}_1\vee \mathcal{P}_2=\mathcal{P}$. Note that
$G^+$ is primitive of type CD on each of its orbits on
$V\Gamma$ and $G$ is vertex-biprimitive. 
}%
\end{example}
\begin{example}
\label{eg:CDHC}
{\rm Let $G=T^8\rtimes K$ where 
$$K=\la (1,2)(3,4)(5,6)(7,8),(1,5)(2,6)(3,7)(4,8),
(1,3,2,4)(5,8,6,7)\ra\cong D_8.$$ Then $K$ has an index 2 subgroup
$K_1=\la (1,2)(3,4)(5,6)(7,8),(1,5)(2,6)(3,7)(4,8)\ra$ which preserves the two
partitions  $\mathcal{P}_1=\{\{1,4\},\{2,3\},\{5,8\},\{6,7\}\}$ and
$\mathcal{P}_2=\{\{1,3\},\{2,4\},\{5,7\},\{6,8\}\}$. Moreover,
$\mathcal{P}_1\vee\mathcal{P}_2= \{\{1,2,3,4\},\{5,6,7,8\}\}$. Thus
$G^+=T^8\rtimes K_1$ is primitive of type HC on each vertex orbit
while $G$ is quasiprimitive but not primitive of type CD on
edges.
}%
\end{example}
\begin{example}
\label{eg:CDG+notqp}
{\rm If $G$ is edge-quasiprimitive but not edge-primitive it is
not even necessary for $G^+$ to be quasiprimitive on each orbit.
For example, let $G=N\rtimes K$ where
$N=T^{4d}$ and $K=(S_d\Wr S_2)\Wr S_2$ such that $K$ preserves the
partition $\{\{1,\ldots,d\},\{d+1,\ldots,2d\},
\{2d+1,\ldots,3d\},\{3d+1,\ldots,4d\}\}$. Now $K$ has an index two
subgroup $K^+=(S_d\Wr S_2)^2$  which has two orbits of size $2d$ on
$\{1,\ldots,4d\}$ and acts imprimitively on each orbit. Then with
$G^+=N\rtimes K^+$, and the two partitions
$\mathcal{P}_1=\{\{1,\ldots,2d\},\{2d+1,\ldots,3d\},\{3d+1,\ldots,4d\}\}$
and
$\mathcal{P}_2=\{\{1,\ldots,d\},\{d+1,\ldots,2d\},\{2d+1,\ldots,4d\}\}$,
Construction \ref{con:SDCDtrans} yields a $G$-edge-quasiprimitive
graph such that $G^+$ is not quasiprimitive on either of its orbits
(since the strips of $N_v$ are not all of equal length). 
}%
\end{example}

\section{Analysing the quasiprimitive and primitive types}
\label{sec:proof}

In this section we determine all the possible types of edge and vertex
actions of edge-quasiprimitive graphs (Theorem \ref{thm:edgeqp}). From
this, after a bit more work we deduce Theorem \ref{thm:eprim}.
By Lemmas \ref{lem:qpfaithful}, \ref{lem:qpbiqp} and
\ref{lem:faithful} there are three types of vertex actions for
$G$-edge-quasiprimitive graphs to consider:   
\begin{itemize}
\item $G$-vertex-intransitive where $G$ acts faithfully and
  quasiprimitively on both orbits; 
\item $G$-vertex-quasiprimitive;
\item $G$-vertex-biquasiprimitive and $G^+$ faithful on each orbit.
\end{itemize}
We go through each of the 8 types of quasiprimitive groups as
possibilities for the edge action and determine if there is a suitable
vertex action in each case.

\begin{lemma}
Let $\Gamma$ be a connected $G$-edge-quasiprimitive graph such that
$G$ is of type HA on edges. Then $\Gamma$ is either a cycle of prime
length or a complete bipartite graph. 
\end{lemma}
\begin{proof}
The unique minimal normal subgroup $N$ of $G$ is elementary abelian
and $G$ is in fact edge-primitive. Since $N$ is
edge-transitive, it is either vertex-transitive or has two orbits. If
$N$ is vertex-transitive, then since $N$ is abelian it acts
regularly on $V\Gamma$ and so $|V\Gamma|=|E\Gamma|$. Hence $\Gamma$ is
a cycle and by the primitivity of $G$ on $E\Gamma$ it follows that
$\Gamma$ has prime length. If $N$ has two orbits, then $G$ is
biquasiprimitive on vertices and so by Lemma \ref{lem:faithful}, either $\Gamma$
is complete bipartite or $N$ acts
faithfully on each orbit. In the latter case, $N$ acts
regularly on each orbit and so there are twice as many vertices as
edges. This contradicts the fact that $\Gamma$ is connected and so
$\Gamma$ is complete bipartite. 
\end{proof}

\begin{lemma}
Let $\Gamma$ be a connected $G$-edge-quasiprimitive graph such that
$G$ is quasiprimitive of type HS or HC on edges. Then $\Gamma$ is a
complete bipartite graph. 
\end{lemma}
\begin{proof}
Let $N_1$ and $N_2$ be the two minimal normal subgroups of $G$. 
Since $G$ is of type HS or HC on edges, it is edge-primitive and so by Lemma
\ref{lem:transitive} either $\Gamma$ is a star (and hence complete bipartite), or $G$ is vertex-transitive. We may assume that $G$ is vertex-transitive. Then by Lemma
\ref{lem:qpbiqp}, $G^{V\Gamma}$ is either quasiprimitive or
biquasiprimitive. If $G^{V\Gamma}$ is 
quasiprimitive then since $G$ has two minimal normal subgroups,
$G^{V\Gamma}$ is of type HS or HC, respectively. Hence $N_1$ and $N_2$
are vertex-regular and so $|E\Gamma|=|V\Gamma|$. Thus $\Gamma$ is a
cycle, contradicting $N_1$ being insoluble. Thus $G^{V\Gamma}$ is
biquasiprimitive. Suppose 
that $\Gamma$ is not complete bipartite. Since
neither $N_1$ nor $N_2$ has an index two subgroup, it follows that
$N_1,N_2\leqslant G^+$ and by Lemma \ref{lem:faithful}, both act
transitively and faithfully on each $G^+$ orbit. Since $N_1$
centralises $N_2$, it follows that $N_1$ and $N_2$ act regularly on
each $G^+$ orbit (\cite[Theorem 4.2A]{dixon}). This implies that there
are twice as many vertices as edges, contradicting $\Gamma$ being
connected. Hence $\Gamma$ is a complete bipartite graph. 
\end{proof}

\begin{lemma}
Let $\Gamma$ be a $G$-edge-quasiprimitive graph which is of type AS on
edges. Then either $G$ is quasiprimitive of type AS on vertices or
$\Gamma$ is bipartite and $G^+$ acts faithfully and quasiprimitively
of type AS on both parts of the bipartition.
\end{lemma}
\begin{proof}
Noticing that any nontrivial normal subgroup of $G$ is almost simple,
the result follows by comparing isomorphism types and Lemmas
\ref{lem:qpfaithful} and \ref{lem:qpbiqp}. 
\end{proof}

Before dealing with the SD and CD cases we need the following lemma.

\begin{lemma}
\label{lem:Ne}
Let $\Gamma$ be a connected $G$-edge-quasiprimitive graph and let $N$
be a normal subgroup of $G$ such that $N\cong T^k$ for some
finite nonabelian simple group $T$. Let $e=\{v,w\}$ be an edge of
$\Gamma$. Then $N_e\neq N_v$.
\end{lemma}
\begin{proof}
Suppose that $N_e=N_v$. Since $N$ is edge-transitive, it has at most
two orbits on vertices. If $N$ is vertex-transitive then
$|V\Gamma|=|E\Gamma|$ and so $\Gamma$ is a cycle. This contradicts
$N\leqslant\Aut(\Gamma)$. Hence $N$ has two orbits on vertices and
$\Gamma$ is bipartite. Let $\Delta_1$ be the bipartite half containing
$v$. Then $|\Delta_1|=|E\Gamma|$.  This contradicts $\Gamma$ being
connected and so $N_e\neq N_v$.
\end{proof}

Next we deal with the SD and CD cases.

\begin{proposition}
\label{prn:SDCD}
Let $\Gamma$ be a $G$-edge-quasiprimitive,
$G$-vertex-transitive connected graph which is not
complete bipartite and such that $G$ is quasiprimitive of type SD or
CD on edges. Let $N\cong T^k$ be the unique minimal normal subgroup of
$G$, let $e=\{v,w\}$ be an edge and let $\mathcal{P}$ be the
partition of the set of $k$ simple direct factors of $N$ given by the
set of supports of the full strips of $N_e$. Then the following all hold.
\begin{enumerate}
\item $\Gamma$ is bipartite and $G^+$ acts faithfully on each
  bipartite half.
\item There exists a nontrivial $G^+$-invariant partition
  $\mathcal{P}_1$ of $\{1,\ldots,k\}$ such that $N_v$ is the product
  of full strips whose supports are the parts of $\mathcal{P}_1$. 
\item There exists a nontrivial $G^+$-invariant partition
  $\mathcal{P}_2$ of $\{1,\ldots,k\}$ such that $N_w$ is the product
  of full strips whose supports are the parts of $\mathcal{P}_2$.
\item $\mathcal{P}_1\vee\mathcal{P}_2=\mathcal{P}$ and $G$
  interchanges $\mathcal{P}_1$ and $\mathcal{P}_2$. 
\item $\Gamma$ is isomorphic to the graph yielded
  by Construction \ref{con:SDCDtrans} using $G$, $\mathcal{P}_1$ and
  $\mathcal{P}_2$. 
\end{enumerate}
\end{proposition}
\begin{proof}
Since $N_e\cong T^l$ for some divisor $l$ of $k$, it does not
have an index two subgroup and so $N_e=N_{vw}$.  Thus $\pi_i(N_v)=T$
for each $i$, and so by a well known lemma, (see for example \cite[p
  328]{Scott}) there exists a partition $\mathcal{P}_1$ of
$\{1,\ldots,k\}$ such that $N_v$ is the product of full strips whose
supports are the parts of $\mathcal{P}_1$. Similarly, $\pi_i(N_w)=T$
for each $i$ and so there exists a partition $\mathcal{P}_2$ of
$\{1,\ldots,k\}$ such that $N_w$ is the product of full strips whose
supports are the parts of $\mathcal{P}_2$. Since $N_e=N_v\cap N_w$,
Lemma \ref{lem:intersectstrips} implies that
$\mathcal{P}_1\vee\mathcal{P}_2=\mathcal{P}$. By Lemma \ref{lem:Ne}
$N_e\neq N_v$, and so $\mathcal{P}_1,\mathcal{P}_2\neq \mathcal{P}$,
hence $\mathcal{P}_1\neq \mathcal{P}_2$. Thus $N$ is
vertex-intransitive and $\mathcal{P}_1,\mathcal{P}_2$ are nontrivial
partitions of $\{1,\ldots,k\}$. Since $N$ is edge-transitive, it
follows that $\Gamma$ is  bipartite with the two bipartite halves
being $N$-orbits. By Lemmas \ref{lem:qpfaithful} and
\ref{lem:faithful}, $G^+$ is faithful on each bipartite half and so
(1) holds. Moreover, $G^+=NG_v=NG_w$ and so $G^+$, $G_v$ and $G_w$ all
induce the same  permutation group on the set of $k$ simple direct
factors of $N$.  Hence $\mathcal{P}_1 $ and $\mathcal{P}_2$ are
$G^+$-invariant and so parts (2) and (3) hold. Furthermore, since $G$
is vertex-transitive there exists $g\in G$ such that $v^g=w$. Thus $G$
interchanges $\mathcal{P}_1$ and $\mathcal{P}_2$ and so part (4)
holds. It remains to prove part (5).   

Conjugating by a suitable element of $\Sym(V\Gamma)$ we may assume
that $N_v$ is a product of straight full strips corresponding to the
parts of $\mathcal{P}_1$. Thus $N_v$ is the subgroup $N_1$ constructed
in Construction \ref{con:SDCDtrans}.  Since $G$ interchanges
$\mathcal{P}_1$ and $\mathcal{P}_2$, it follows that $G_v\leqslant
N_G(N_v)\leqslant G^+$. Since $G^+=NG_v$ and $N_v$ is selfnormalising
in $N$, it follows that $G_v=N_G(N_v)$. Thus $G_v$ is the subgroup $H$
given in Construction \ref{con:SDCDtrans}.  Letting $g\in G_e$ which
interchanges $v$ and $w$ and hence $\mathcal{P}_1,\mathcal{P}_2$, it
follows that $\Gamma\cong \Cos(G,H,HgH)$, the graph constructed in
Construction \ref{con:SDCDtrans}.  Thus part (5) holds.
\end{proof} 

We have the following corollaries if $G$ is edge-primitive. 

\begin{corollary}
\label{cor:SDeprim}
Let $\Gamma$ be a connected $G$-edge-primitive graph which is not
complete bipartite such that $G$ is primitive of type {\rm SD} on
edges. Then $\Gamma$ is bipartite and $G^+$ is faithful and
quasiprimitive of type {\rm CD} on each bipartite half.
\end{corollary}
\begin{proof}
Since $G^{E\Gamma}$ is primitive of type SD it follows that
$\mathcal{P}=\{\{1,\ldots, k\}\}$ and $G$ acts primitively on the set
of $k$ simple direct factors of $N$. Since $G^+\norml G$ it follows
that $G^+$ acts transitively on the the set of simple direct factors of
$N$. Hence $N$ is a minimal normal subgroup of $G^+$ and so $G^+$ acts
faithfully and quasiprimitively on each orbit. By Lemma \ref{lem:Ne},
$N_e<N_v$ and so this action is of type CD. 
\end{proof}
\begin{corollary}
\label{cor:CDeprim}
Let $\Gamma$ be a connected $G$-edge-primitive graph which is not
complete bipartite such that $G$ is primitive of type {\rm CD} on
edges. Then $\Gamma$ is bipartite and $G^+$ is faithful and
quasiprimitive of type {\rm CD} on each bipartite half.
\end{corollary}
\begin{proof}
Let $N\cong T^k$ be the unique minimal normal subgroup of $G$.
Let $e$ be an edge and let $\mathcal{P}$ be the symstem of
imprimitivity of $\{1,\ldots,k\}$ given by the set of supports of the
strips of $N_e$. Since $G$ is primitive of type CD on edges it follows
that for $P\in\mathcal{P}$, $G_P$ acts primitively on $P$. Also
$|G_P:G_P^+|\leq 2$.  If $|G_P:G_P^+|=1$ then $G_P^+$ acts primitively
on $P$. However, by Proposition \ref{prn:SDCD},
$\mathcal{P}=\mathcal{P}_1\vee\mathcal{P}_2$ where $\mathcal{P}_1$ and
$\mathcal{P}_2$ are preserved by $G^+$. Hence $P$ is a union of blocks
of $\mathcal{P}_1$, contradicting $G_P^+$ acting primitively on
$P$. Thus $|G_P:G_P^+|=2$ and so $G^+$ is transitive on
$\mathcal{P}$. Moreover, as $G_P$ is primitive on $P$ it
follows that $G_P^+$ is transitive on $P$ and so $G^+$ is transitive
on the set of $k$ simple direct factors on $N$. Hence $N$ is a minimal
normal subgroup of $G^+$ and so $G^+$ acts faithfully and quasiprimitively
on each orbit. By Lemma \ref{lem:Ne}, $N_e<N_v$ and so this action is
of type CD.
\end{proof}

Next we investigate the case where $G$ is of type PA on edges.
\begin{lemma}
\label{lem:Nvneq1}
Let $\Gamma$ be a $G$-edge-quasiprimitive connected graph such that
$G$ is of type PA on edges. Let $N$ be the unique minimal normal
subgroup of $G$. Then $N_v\neq 1$.
\end{lemma}
\begin{proof}
Since $G$ is quasiprimitive of type PA on edges we have that 
$N_e\neq 1$.  Suppose that $N_v=1$.  Then 
$|V\Gamma|\geq|N|>|E\Gamma|$, contradicting $\Gamma$ being connected.
Thus $N_v\neq 1$.  
\end{proof}
\begin{corollary}
Let $\Gamma$ be a $G$-edge-quasiprimitive connected graph such that
$G$ is of type {\rm PA} on edges. Suppose that $G$ is
vertex-quasiprimitive. Then the quasiprimitive type of $G^{V\Gamma}$
is {\rm SD, CD} or {\rm PA}.
\end{corollary}
\begin{proof}
Let $N$ be the unique minimal normal subgroup of $G$. By Lemma
\ref{lem:Nvneq1}, $N_v\neq 1$ and so $G$ is not of type TW on
vertices. Since $G$ has a unique minimal normal subgroup which is not
elementary abelian or simple, it follows that $G^{V\Gamma}$
is of type SD, CD or PA.
\end{proof}
\begin{corollary}
Let $\Gamma$ be a $G$-edge-quasiprimitive connected graph such that
$G$ is of type {\rm PA} on edges. Suppose that $G$ is
vertex-intransitive. Then the quasiprimitive type of $G$ on
each of its orbits is {\rm SD, CD} or {\rm PA}.
\end{corollary}
\begin{corollary}
Let $\Gamma$ be a $G$-edge-quasiprimitive connected graph such that
$G$ is of type {\rm PA} on edges. Suppose that $G$ is
vertex-biquasiprimitive and $G^+$ is quasiprimitive on each
orbit. Then the quasiprimitive type of $G^+$ on each of its orbits is
{\rm HS, HC, SD, CD} or {\rm PA}. 
\end{corollary}

Collecting together our results  we have the following two
theorems. We split the statements into the vertex-transitive and
vertex-intransitive cases.

\begin{theorem}
\label{thm:edgeqp}
Let $\Gamma$ be a $G$-edge-quasiprimitive,
$G$-vertex-transitive  connected graph of valency
at least three such that $G^{E\Gamma}$ is of type $X$. Then one of the
following holds.
\begin{enumerate}
\item $\Gamma$ is a complete bipartite graph.
\item $X\in\{\mathrm{SD, CD}\}$ and $\Gamma$ can be obtained from
  Construction \ref{con:SDCDtrans}.  
\item $X=\mathrm{PA}$ and $G$ is quasiprimitive on $V\Gamma$ of
  type {\rm SD, CD} or {\rm PA}. 
\item $X=\mathrm{PA}$ and $\Gamma$ is bipartite, such that $G^+$
  is faithful and quasiprimitive on each of its orbits of type
  $Y\in\{\mathrm{HS, HC, SD, CD, PA}\}$.  
\item $X=\mathrm{PA}$, $\Gamma$ is bipartite, and
  $G^+$ is not quasiprimitive on either orbit.
\item $X=\mathrm{AS}$ and either $G^{V\Gamma}$ is quasiprimitive
  of type {\rm AS} or $\Gamma$ is bipartite and $G^+$ is faithful and
  quasiprimitive of type {\rm AS} on each of its orbits. 
\item $X={\mathrm TW}$.
\end{enumerate}
Moreover, examples occur in all cases. 
\end{theorem}

Examples \ref{eg:SD} and \ref{eg:CDCD} provide edge-primitive
examples for case (2), Construction \ref{con:PAPA} gives examples for
case (3) where $G$ is primitive of type PA on vertices, Construction 
\ref{con:PAbiqp} gives examples where $G^+$ is primitive of type PA on
each orbit and Section \ref{sec:eg} gives many example for case
(6). An edge-primitive example for case (5) is given by Remark
\ref{rem:notqp}. Examples of edge-quasiprimitive but not
edge-primitive are given in Section \ref{sec:qpeg}. 
 
If $G$ is edge-primitive we can sometimes deduce more information. For
example, we can eliminate  $X=\mathrm{TW}$.

\begin{proposition}
\label{prn:TW}
Let $\Gamma$ be a $G$-edge-primitive graph such that
$G$ is of type $\mathrm{TW}$ on edges. Then $\Gamma$ is a complete bipartite
graph.
\end{proposition}
\begin{proof}
Let $\Gamma$ be a $G$-edge-primitive graph such that $G$ is of type TW
on edges. Let $N$ be the unique minimal normal subgroup of $G$. Then
$N=T_1\times\cdots\times T_k$ with each $T_i\cong T$ for some finite
nonabelian simple group $T$ and $G=N\rtimes G_e$. Moreover, $G_e$ acts
transitively by conjugation on the set of $k$ simple direct factors of
$N$. Let $(G_e)_1$ be the normaliser in $G_e$ of $T_1$ and
$\varphi:(G_e)_1\rightarrow \Aut(T)$ be the homomorphism induced by
the action of $(G_e)_1$ on $T_1$ by conjugation. By Lemma
\ref{lem:TWmax}, since $G_e$ is maximal in $G$ we have that
$\Inn(T)\leqslant\varphi((G_e)_1)$ and $\varphi$ extends to no
overgroup of $(G_e)_1$ in $G_e$. Since $G$ is arc-transitive it
follows that $G_{vw}$ is an index two subgroup of $G_e$. There are two
cases to consider: $(G_e)_1\cap G_{vw}$ is an index two subgroup of
$(G_e)_1$, or $(G_e)_1\leqslant G_{vw}$. 

Suppose that $(G_e)_1\cap G_{vw}$ is an index two subgroup of
$(G_e)_1$. Then $G_{vw}$ acts transitively on the set of $k$ simple direct
factors of $N$. Since $\Inn(T)$ does not have an index two subgroup,
it follows that $\Inn(T)\leqslant \varphi(G_{vw})$. Suppose that there
exists $R$ with $(G_e)_1\cap G_{vw}\leqslant R\leqslant G_{vw}$ such
that $\varphi$ extends to $R$. Then $\varphi$ would extend to $\la
(G_e)_1,R\ra\leqslant G_e$. 
Since $\varphi$ does not extend to any overgroup of $(G_e)_1$ in $G_e$ it
follows that $R\leqslant (G_e)_1$ and so $R=G_{vw}\cap (G_e)_1$. Thus
by Lemma \ref{lem:TWmax}, $(G_e)_1\cap G_{vw}$ is maximal in 
$N\rtimes ((G_e)_1\cap G_{vw})$. Since $G_{vw}$ normalises $N_v$ and $N_w$,
it follows that $N_v=N_w=1$. Thus $|V\Gamma|=|N|$ or $2|N|$. However,
$|E\Gamma|=|N|$ and so $|V\Gamma|=|N|$ and $\Gamma$ is a cycle. This
contradicts $G$ being insoluble and so $(G_e)_1\leqslant G_{vw}$.

Since $(G_e)_1\leqslant G_{vw}$ it follows that $G_{vw}$ has two equal
sized orbits on the set of $k$ simple direct factors of $N$. Without
loss of generality we may suppose that these are
$\{T_1,\ldots,T_{k/2}\}$ and $\{T_{k/2+1},\ldots,T_k\}$ and note that
they are interchanged by elements of $G_e$ not in $G_{vw}$. Moreover,
$(G_e)_1$ normalises $N_v$. Since $\varphi((G_e)_1)$ contains
$\Inn(T)$ it follows that the projection of $N_v$ onto the first
simple direct factor of $N$ is either trivial or equal to $T$. Thus
$N_v$ is a subdirect product of either $T_{k/2+1}\times \cdots\times
T_k$ or $N$. If $N_v\leq T_{k/2+1}\times \cdots\times T_k$ then
$N_w\leq T_1\times\cdots \times T_{k/2}$. Moreover, $G_e$ normalises
$\la N_v,N_w\ra$ and so by the maximality of $G_e$ in $G$ we have 
$\la N_v,N_v\ra=N$.  Thus $N_v=T_{k/2+1}\times \cdots\times T_k$, and so
$N$ has two orbits on vertices and is unfaithful on each. Hence by Lemma
\ref{lem:faithful}, $\Gamma$ is a complete bipartite graph. Thus we
are left to consider the case where $N_v$ is a subdirect product of $N$. 
Thus there exists a partition $\mathcal{P}$ of $\{1,\ldots,k\}$ such that
$N_v=\prod_{I\in\mathcal{P}}T_I$ where $T_I$ is a diagonal subgroup of
$\prod_{i\in I}T_i$. Since $(G_e)_1\leqslant G_{vw}\leqslant G_e$ it
follows from Lemma \ref{lem:TWmax} that $G_{vw}$ is a maximal subgroup
of $(T_1\times \cdots\times T_{k/2})\rtimes G_{vw}$. Hence 
$N_v\cap (T_1\times \cdots\times T_{k/2})=1$. Similarly, $G_{vw}$ is
maximal in $(T_{k/2+1}\times \cdots\times T_k)\rtimes G_{vw}$ and so
$N_v \cap (T_{k/2+1}\times \cdots\times T_k)=1$. It follows that each
$I\in\mathcal{P}$ is split equally between $\{1,\ldots,k/2\}$ and
$\{k/2+1,\ldots,k\}$. However, since $G_{vw}$ normalises $N_v$ and
$M=T_1\times \cdots\times T_{k/2}$ it follows that $G_{vw}$ normalises
the projection of $N_v$ onto $M$. Thus $|I|=2$, as $G_{vw}$ normalises
no proper nontrivial subgroup of $M$. Hence $|N_v|=|T|^{k/2}$ and
$|V\Gamma|=|T|^{k/2}$ or $2|T|^{k/2}$. The first case is not possible
as $|V\Gamma|^2=|E\Gamma|$, a contradiction. Hence we have the
second. This implies that $\Gamma$ is complete bipartite and we are
done. 
\end{proof}

We can also deduce more information when $X=\mathrm{PA}$.  

\begin{lemma}
\label{lem:PA}
Let $\Gamma$ be a $G$-edge-primitive graph such that $G$ is of type
$\mathrm{PA}$ 
on edges and $\Gamma$ is not complete bipartite. Then one of the
following holds: 
\begin{enumerate} 
\item $G$ is quasiprimitive on vertices of type $\mathrm{PA}$;
\item  $G$ is biquasiprimitive and $G^+$ is quasiprimitive of type
  $\mathrm{PA}$
  on each bipartite half; 
\item $G$ is biquasiprimitive and $G^+$ is not quasiprimitive on
  either bipartite half.
\end{enumerate}
\end{lemma}
\begin{proof}
Let $N$ be the unique minimal normal subgroup of $G$. Then $N=T^k$ for
some finite nonabelian simple group $T$ and $k \geq 2$. Also given an
edge $e=\{v,w\}$ we have $N_e=R^k$ for some proper nontrivial subgroup $R$ of
$T$.  Since $G^{E\Gamma}$ is primitive, there exists an almost simple
group $A$ with socle $T$ and maximal subgroup $H$ such that $H\cap
T=R$. Suppose that $|R|=2$. Then $H=C_A(z)$ and $R=C_T(z)$, where $z$
is the involution which generates $R$. However, $4$ divides $|T|$ and
so either $z$ is contained in a cyclic group of order 4 or an
elementary abelian group of order 4, a contradiction. Thus $|R|>2$. It
follows that $N_e$ does not have an index 2 subgroup and so
$N_{vw}=N_e$. Hence $R^k\leqslant N_v$. Thus for each $i$ such that
$\pi_i(N_v)=T$, we have that $N_v$ contains the $i^{\mathrm{th}}$
factor of $N$. Since $\Gamma$ is not complete bipartite, Lemma
\ref{lem:faithful} implies that $N$ is faithful on each of its orbits on
$V\Gamma$, and so $N$ cannot contain any of its simple direct
factors. Thus $\pi_i(N_v)\neq T$ for all $i$.
Hence if $G$ is quasiprimitive on $V\Gamma$, this implies that $G$ is
of type PA on vertices and we have case (1). If $G$ is
biquasiprimitive on vertices and $G^+$ is transitive on the set of
simple direct factors of $N$ then we have that $G^+$ is quasiprimitive
of type PA on each of its orbits and we have case (2). If $G^+$ has
two orbits on the set of simple direct factors of $N$ then $G^+$ has
two minimal normal subgroups contained in $N$. Since $N_v$ does not
project onto $T$ in any coordinate, it follows that $G^+$ is not
quasiprimitive on either orbit and so case (3) holds.
\end{proof}

Note that when $G$ is biprimitive on vertices, $G^+$ is primitive on
each bipartite half. Hence Lemma \ref{lem:PA} combined with
Theorem \ref{thm:edgeqp}, Corollaries \ref{cor:SDeprim} and
\ref{cor:CDeprim}, and Proposition \ref{prn:TW} yields Theorem
\ref{thm:eprim}.   

We complete this section by reducing the study of edge-primitive graphs
of type PA to the 
study of edge-primitive graphs of type AS. Before doing so we need to
establish some notation. Let $T$ be a nonabelian simple group and
let $G$ be a subgroup of $\Aut(T)\Wr S_k$ for some $k\geq 2$, which
contains $N=T_1\times\cdots\times T_k$ where each $T_i\cong T$, and
such that $G$ induces a transitive subgroup of $S_k$ on the set of $k$
simple direct factors of $N$. Let $G_1$ be the normaliser in $G$ of
$T_1$. Then $G_1=G\cap (\Aut(T)\times (\Aut(T)\Wr S_{k-1}))$ and there
exists a projection $\pi_1:G_1\rightarrow \Aut(T)$.  Let
$B=\pi_1(G_1)$. By \cite[(2.2)]{kovacs}, conjugating by a suitable element of
$\Aut(T)\Wr S_k$ we may have chosen $G$ such that $G\leqslant B\Wr
S_k$.  We call $B$ the \emph{group induced by $G$}.   

\begin{proposition}
\label{prn:PAtoAS}
Suppose that $\Gamma$ is a $G$-edge-primitive graph such that
$G^{E\Gamma}$ is of type PA, and let $e=\{v,w\}$ be an edge. Let
$N=\soc(G)\cong T^k$ for some finite nonabelian simple group $T$ and
$k$ a positive integer at least two. Suppose that $G$ induces the
primitive almost simple group $B$ with socle $T$, and that $G_e$ and
$G_v$ induce the subgroups $E$ and $H$ of $B$ respectively. Then there
exists a $B$-edge-primitive graph with edge-stabiliser $E$ and
vertex-stabiliser $H$.
\end{proposition}
\begin{proof}
Since $G$ is a primitive group of type PA on $E\Gamma$, we have that
$G\leqslant B\Wr S_k$ and $G_e=G\cap (E\Wr S_k)$ where
$\pi_1((G_1)_e)=E$ is a maximal subgroup of $B$. Let
$A=\pi_1((G_1)_{vw})$ and $H=\pi_1((G_1)_v)$. Note
that $H\cap E=A$ and $H$ is a proper subgroup of $B$. Since $G$ is
arc-transitive, $|G_e:G_{vw}|=2$ and so $|(G_1)_e:(G_1)_{vw}|\leq 2$.
Thus $|E:A|\leq 2$. If $E=A$ then $E\leq H$. However, by the
maximality of $E$ this implies that $E=H$ and so $G_v$ is contained in
some $G$-conjugate of $G_e$. This contradicts the fact that there are
more edges than vertices and so $|E:A|=2$. For the same reason $A<H$.
Let $\sigma\in (G_1)_e\backslash (G_1)_{vw}$. Then
$g=\pi_1(\sigma)\in E\backslash A$ and $\Cos(G,H,HgH)$ is a
$G$-edge-primitive graph with edge stabiliser $E$ and vertex
stabiliser $H$.
\end{proof}

\section{Quasiprimitive examples}
\label{sec:qpeg}

In this section we construct examples of edge quasiprimitive graphs
where the types of actions do not occur in the edge primitive case.
\begin{example}
\label{eg:PASD}
(Quasiprimitive PA on edges and primitive SD on vertices)
{\rm Let $G=T\Wr S_2$ for some finite nonabelian simple group $T$ and let
  $H=\{(t,t)\mid t\in T\}\times \la \sigma\ra$, where $\sigma$
  interchanges the two simple direct factors of $N=T^2\norml G$. Let
  $x\in T$ be of order two and let $g=(1,x)\in G$. Then
  $$H^g=\{(t,t^x)\mid t\in T\} \times \la (x,x)\sigma\ra
=\{(t,t^x)\mid t\in T\}\rtimes \la \sigma\ra$$ 
and 
$H\cap H^g=\{(t,t)\mid t\in C_T(x)\}\times \la \sigma\ra$. Let
  $\Gamma=\Cos(G,H,HgH)$. Then $G$ is vertex-primitive of type 
 SD.  Let $e=\{H,Hg\}$, an edge of $\Gamma$. Then 
$G_e=\la H\cap H^g,g\ra$. Since $x\in C_T(x)$, it follows that 
$G_e=\{(x^it,x^jt)\mid t\in C_T(x); i,j\in\{0,1\}\}\rtimes \la \sigma\ra$ 
and so $G$ is quasiprimitive of type PA on edges.
}%
\end{example}

\begin{example}
\label{eg:PACD}
(Quasiprimitive PA on edges and primitive CD on vertices)
{\rm Let $\sigma=(1,2,3,4)$ and $G=T^4\rtimes \la \sigma\ra$ for some
finite nonabelian simple group $T$. Let 
$H=\{(t,s,t,s)\mid s,t\in T\}\rtimes \la \sigma\ra$ and $g=(x,x,1,1)$
where $x\in T$ has order two. Then $g^2 \in H$, $g\notin N_G(H)$ and
$\la H,g\ra=G$. Let $\Gamma=\Cos(G,H,HgH)$. Then $G$ is primitive of
type CD on vertices. Let $v$ be the vertex corresponding to $H$ and
$w$ the vertex corresponding to $Hg$. Then
$G_w=H^g=\{(t^x,s^x,t,s)\mid t,s\in T\}\rtimes \la\sigma\ra$ and so for
the edge $e=\{v,w\}$ we have 
$G_e=\{(tx^i,sx^j,tx^k,sx^l)\mid t,s\in C_T(x);i+j+k+l\equiv 0\pmod 2\}\rtimes
\la\sigma\ra$. Thus $G$ is quasiprimitive of type PA on edges. 
}%
\end{example}

We now give examples in the bipartite case.
\begin{example}
\label{eg:PAHS}
(Quasiprimitive PA on edges and $G^+$ primitive HS on each vertex orbit)
{\rm Let $T$ be a nonabelian simple group and let $x\in T$ have order 2.
Let $G=T\Wr S_2$ and $H=\{(t,t)\mid t\in T\}$. Let $g=(x,1)\sigma$
where $\sigma$ interchanges the two simple direct factors of
$N=T^2\norml G$. Then $g^2\in H$, $g\notin N_G(H)$ and 
$\la H,g\ra=G$.  Let $\Gamma=\Cos(G,H,HgH)$. Then $G^+=N$ has two
orbits on vertices. Let $v$ be the vertex
corresponding to $H$ and  $w$ be the vertex corresponding to $Hg$.
Then $G_v=H$ and $G_w=H^g=\{(t,t^x)\mid t\in T\}$. Thus $G^+$ is primitive
of type HS on each orbit. Moreover, $e=\{v,w\}$ is an edge and
$G_e=\{(t,t)\mid t\in C_T(x)\}\times \la g\ra$. Thus $G$ is
quasiprimitive of type PA on edges.
}%
\end{example}

\begin{example}
\label{eg:PAHC}
(Quasiprimitive PA on edges and $G^+$ primitive HC on each vertex orbit)
{\rm Let $\sigma=(1,2,3,4)$ and let $G=T\Wr \la \sigma\ra$ for some finite
nonabelian simple group $T$. Let
$H=\{(t,t,s,s)\mid t,s \in T\} \rtimes \la \sigma^2\ra$ and let $x\in T$ of
order 2. Let $g=(1,x,1,x)\sigma$. Then $g^2\in H$, $g\notin N_G(H)$
and $\la H,g\ra=G$. Let $\Gamma=\Cos(G,H,HgH)$, let $v$ be the vertex
corresponding to $H$ and $w$ the vertex corresponding to $Hg$. Then
$\Gamma$ is bipartite with $G^+=T^4\rtimes \la \sigma^2\ra$ and
$e=\{v,w\}$ is an edge. Moreover, $G^+$ is primitive of type HC on
each orbit. Now $G_v=H$ and 
$G_w=H^g=\{(t,t^x,s,s^x)\mid t,s \in T\} \rtimes \la \sigma^2\ra$. Thus
$G_e=\{(t,t,s,s)\mid t,s \in C_T(x)\}\rtimes \la g\ra$. Hence $G$ is
quasiprimitive of type PA on edges.
}%
\end{example}

\begin{example}
\label{eg:PAGplusSD}
(Quasiprimitive PA on edges and $G^+$ primitive SD on each vertex orbit)
{\rm Let $T$ be a nonabelian simple group with outer automorphism $\tau$ of
order two. Let $G=(T\times T)\rtimes \la (1,\tau),\sigma\ra$ where
$\sigma$ interchanges the two minimal normal subgroups of $N=T^2$. Let
$H=\{(t,t)\mid t\in \la T,\tau\ra\}\times \la \sigma\ra$ and let 
$g=(1,\tau)$. Then $g^2\in H$, $g\notin N_G(H)$ and $\la
H,g\ra=G$. Thus we can define the graph $\Gamma=\Cos(G,H,HgH)$. Let $v$
be the vertex corresponding to $H$ and $w$ the adjacent vertex
corresponding to $Hg$. Then $G_v=H$ and 
$G_w=H^g=\{(t,t^\tau)\mid t\in \la T,\tau\ra\}\rtimes \la
\sigma\ra$. Hence $G^+=T^2\rtimes \la (\tau,\tau),\sigma\ra$ acts
primitively of type SD on each orbit. Let $e=\{v,w\}$. Then 
$G_e=\{(t,t)\mid t\in C_T(\tau)\}\rtimes \la (1,\tau),\sigma\ra$ and
so $G$ is quasiprimitive of type PA on edges. 
}%
\end{example}

\begin{example}
\label{eg:PAGplusCD}
(Quasiprimitive PA on edges and $G^+$ primitive CD on each vertex
orbit)
{\rm Let $\sigma=(1,2,3,4)$, $T$ be a finite nonabelian simple group and
$\tau$ an outer automorphism of $T$ of order two. Let 
$G=T^4\rtimes \la (\tau,\tau,1,1),\sigma\ra$ and 
$H=\{(t,s,t,s)\mid t,s \in T\}\rtimes \la
(\tau,1,\tau,1),\sigma\ra$. Then letting $g=(\tau,\tau,1,1)$ we
see that $g^2\in H$, $g\notin N_G(H)$ and $\la H,g\ra=G$. Thus we can
define the graph $\Cos(G,H,HgH)$. Then 
$G^+=T^4\rtimes \la(\tau,1,\tau,1),\sigma\ra$ acts primitively
of type CD on each vertex orbit.  Let $v$ be the vertex corresponding
to $H$ and $w$ be the adjacent vertex corresponding to $Hg$. Then
$G_v=H$ and $G_w=H^g=\{(t^\tau,s^\tau,t,s)\mid t,s\in T\}\rtimes 
\la (\tau,1,\tau,1),\sigma\ra$. Thus 
$G_e=\{(t,s,t,s)\mid t,s\in C_T(\tau)\}\rtimes \la
(\tau,\tau,1,1),\sigma\ra$ and so $G$ acts quasiprimitively of type PA
on edge. 
}%
\end{example}

\begin{construction}
\label{con:TW}
(Quasiprimitive of type TW on edges and $G^+$ primitive of type PA on
  both orbits.)
{\rm Let $T$ be a finite nonabelian simple group  with maximal subgroup $R$
and suppose that there exists and outer automorphism $\tau$ or order
two such that $R\cap  R^{\tau}=1$. A suitable choice of $T$ and $R$ is
  $\PSL(2,29)$ and $A_5$ respectively.
Let 
$G=\la T^k,(\tau,\ldots,\tau)\ra \rtimes S_k$ and $H=R^k\rtimes S_k$.
Then if $g=(\tau,\ldots, \tau)$ we have $g^2\in H$, $g\notin N_G(H)$
  and $\la H,g\ra=G$. Hence $\Gamma=Cos(G,H,HgH)$ is a
  $G$-arc-transitive connected graph. Moreover, $\la g\ra\times S_k$
  is the stabiliser of an edge. Thus letting $N=\soc(G)= T^k$ we
  have that $N$ acts regularly on $E\Gamma$ and so $G$ is
  quasiprimitive of type TW on edges. Note that $G^{E\Gamma}$ is not
  primitive as an edge stabiliser is not maximal. Furthermore,
  $\Gamma$ is bipartite with $G^+=T^k\rtimes S_k$ acting primitively
  of type PA on both orbits.
}%
\end{construction}

\section{Edge-primitive groups with socle $\PSL(2,q)$}
\label{sec:PSL}

The following theorem of Dickson \cite{dickson} determines the maximal subgroups of
$\PSL(2,q)$. 

\begin{theorem}
\label{thm:maxPSL}
Let $p$ be a prime, $f$ a positive integer and $q=p^f$. Then the
conjugacy classes of maximal subgroups of $\PSL(2,q)$ are as follows:
\begin{enumerate} 
\item one class of subgroups isomorphic to 
      $[q]\rtimes C_{(q-1)/(2,q-1)}$,
\item one class of subgroups isomorphic to $D_{2(q-1)/(2,q-1)}$, if
  $q\notin\{5,7,9,11\}$,
\item one class of subgroups isomorphic to $D_{2(q+1)/(2,q-1)}$, if
  $q\notin \{7,9\}$,
\item two classes of subgroups isomorphic to $A_5$, if 
$q\equiv \pm 1 \pmod {10}$, and $\F_q=\F_p[\sqrt 5]$,
\item two classes of subgroups isomorphic to  $S_4$, if 
$q=p\equiv \pm 1 \pmod 8$, 
\item one class of subgroups isomorphic to $A_4$, if 
$q=p\equiv 3,5,13,27,37 \pmod {40}$.
\item two classes of subgroups isomorphic to $\PGL(2,p^{f/2})$ when
  $p$ odd,
\item one class of subgroups isomorphic to $\PSL(2,p^m)$ where $f/m$
  an odd prime or $p=2$ and $m\geq 2$.
\end{enumerate}
\end{theorem}

We also have the following theorem about maximal
subgroups of almost simple groups with socle $\PSL(2,q)$. 
\begin{theorem}\cite[Theorem 1.1]{maxsubs}
\label{thm:novelties}
Let $T=\PSL(2,q)\leqslant G\leqslant \PGammaL(2,q)$ and let $E$ be a
maximal subgroup of $G$ which does not contain $T$. Then either 
$E\cap T$ is maximal in $T$, or we have one of the following cases.
\begin{enumerate}
\item $G=\PGL(2,7)$ and $E=N_G(D_6)=D_{12}$.
\item $G=\PGL(2,7)$ and $E=N_G(D_{8})=D_{16}$.
\item $G=\PGL(2,9)$, $M_{10}$ or $\PGammaL(2,9)$ and $E=N_G(D_{10})$
\item $G=\PGL(2,9)$, $M_{10}$ or $\PGammaL(2,9)$ and $E=N_G(D_{8})$.
\item $G=\PGL(2,11)$ and $E=N_G(D_{10})=D_{20}$.
\item $G=\PGL(2,q)$, $q=p\equiv \pm 11, \pm 19\pmod{40}$ and $E=N_G(A_4)=S_4$.
\end{enumerate}
\end{theorem}

The following lemma shows that except for the 8 exceptions in Theorem
\ref{thm:novelties}, we can restrict our attention to searching for
$G$-arc-transitive $G$-edge primitive graphs with $G=\PSL(2,q)$.
\begin{lemma}
Let $\Gamma$ be a nontrivial $G$-edge-primitive connected graph with
$T=\PSL(2,q)\norml G\leqslant \PGammaL(2,q)$. Let $E$ be the
stabiliser in $G$ of an edge of $\Gamma$. If $E\cap T$ is maximal in
$T$ then $T$ is arc-transitive and edge-primitive.
\end{lemma}
\begin{proof}
Since $G$ is edge-primitive and $T\norml G$ it follows that $T$ acts
transitively on the set of edges with edge stabiliser $E\cap T$. Hence
$T$ is edge-primitive.  Since $\Gamma$ is $G$-arc-transitive, $\Gamma$
is not a star and so by Lemma \ref{lem:arctrans}, $\Gamma$ is also
$T$-arc-transitive.
\end{proof}

We have the following proposition.
\begin{proposition}
\label{prn:containment}
Let $G=\PSL(2,q)$, where $q=p^f$, $E$ be a maximal subgroup of $G$ and
$H$ be a subgroup of $G$ such that $A=H\cap E$ is an index two
subgroup of $E$ and a proper subgroup of $H$. Suppose $G$ is not
2-transitive on the set of cosets of $H$. Then one of the following
holds. 
\begin{enumerate}
\item $q=p\equiv \pm 1,\pm9\pmod{40}$, $E=S_4$, $H=A_5$ and $A=A_4$.
\item $q=17$, $E=D_{16}$, $H=S_4$ and $A=D_{8}$.
\item $q=19$, $E=D_{20}$, $H=A_5$ and $A=D_{10}$.
\item $q=25$, $E=D_{24}$, $H=\PGL(2,5)$ and $A=D_{12}$.
\end{enumerate}
In the first case, given $E$ there are two choices for $H$ and
these are conjugate in $T$. In the last three cases, given $E$ there
are four choices for $H$ and these come in conjugate pairs.
\end{proposition}
\begin{proof}
We work our way through the list of maximal subgroups of $G$ given in
Theorem \ref{thm:maxPSL}.  We note first that $E$ cannot be $A_5$,
$A_4$, $\PSL(2,p^m)$ for $p^m\neq 2$, or $[q]\rtimes C_{q-1}$ for $q$
even, as these groups do not have an
index 2 subgroup. Furthermore, $E\neq [q]\rtimes C_{(q-1)/2}$ for $q$
odd as the only possible index 2 subgroup is $[q]\rtimes C_{(q-1)/4}$
which is only contained in $E$.

Suppose next that $E=D_{2(q-1)/(2,q-1)}$ and note that $q\notin \{5,7,9,11\}$.
Then $A= C_{(q-1)/(2,q-1)}$, the 
stabiliser of two points of the projective line, is an index two
subgroup of $E$. The only possibility for $H$ is a subgroup isomorphic
to $[q]\rtimes  C_{(q-1)/(2,q-1)}$, but in this case the action of $G$
is 2-transitive.
If $(q-1)/(2,q-1)$ is even then $E$ also contains two subgroups
isomorphic to $D_{(q-1)/2}$ which are conjugate in $\PGL(2,q)$ but not
$\PSL(2,q)$. The restrictions on $q$ imply that
$(q-1)/(2,q-1)\geq 6$ and so if $A\cong D_{(q-1)/2}$ then $A$ is
not contained in a $D_{q+1}$. Furthermore, $A$ is not contained in an
$A_4$.  If $A$ is contained in an $A_5$ then $(q-1)/2=6$ or 10. The
first implies that $q=13$, but $\PSL(2,13)$ does not contain an $A_5$
while the second implies that $q=21$, a contradiction. Thus $A$ is not
contained in an $A_5$. If $A$ is contained in an $S_4$ then
$(q-1)/2=6$ or $8$. Again the first is not possible as $\PSL(2,13)$
does not contain an $S_4$ and so $q=17$.  Since $D_8$ is maximal in
$S_4$, it follows that in this case we have $H\cong S_4$. Counting
again shows that given $A$ there are two choices for $H$ and these are
conjugate in $T$. The two nonconjugate choices for $A$ give us two
nonconjugate pairs of choices for $H$. Thus we are in case (2). If
  $A\leqslant \PGL(2,p^{f/2})$ then $(q-1)/2$ divides either
  $2(p^{f/2}-1)$ or $2(p^{f/2}+1)$. Since $q-1=(p^{f/2}-1)(p^{f/2}+1)$
either $p^{f/2}-1$ or $p^{f/2}+1$ divides 4. Thus $p^{f/2}=3$ or
5. Since $q\neq 9$ this give us 
$D_{12}\leqslant \PGL(2,5)\leqslant \PSL(2,25)$. Counting again gives
that there are two choices for $H$ and these are conjugate in
$T$. Again the two nonconjugate choices for $A$ give nonconjugate
pairs of choices for $H$ and we have case (4). 
If $A\leqslant \PSL(2,p^{f/r})$ for $r\geq 3$ then $(q-1)/2$ divides
either $p^{f/r}-1$ or $p^{f/r}+1$. Since $r\geq 3$ we have
$p^f-1>2(p^{f/r}\pm1)$ and so this is not possible. 

Next let $E=D_{2(q+1)/(2,q-1)}$ with
$q\notin\{7,9\}$. One choice for $A$ is $C_{(q+1)/(2,q-1)}$. If $q=5$
then $A=C_3$ and so $H\cong A_4$. However, in this case $G$ is 2-transitive on the cosets of $H$. Thus $(q+1)/(2,q-1) \geq 6$ and so there is no
possibility for $H$. 
If $(q+1)/2$ is even then  $A$ can also be 
one of the two choices of $D_{(q+1)/(2,q-1)}$ which are conjugate in
$\PGL(2,q)$ but not $\PSL(2,q)$. Note then that 
$q\geq 11$ and so $(q+1)/(2,q-1)\geq 6$.  Thus $A$ is not contained
  in a $A_4$.
Since $(q+1)/2\geq 6$ does not divide $q-1$ it
  follows that $A$ is not contained in $D_{q-1}$. Now 
$A\leqslant A_5$, if and only if $(q+1)/2=10$ or $6$. For $A_5$ to be
  a subgroup of $G$ we require that $q=11$ or 19. We do not have the first case as this yields a 2-transitive group. There are then two choices for
  $H$ and these are conjugate in $T$. Moreover the 
  two nonconjugate choices for $A$ gives nonconjugate pairs of choices
  for $H$ and we have case (3). To have 
$A\leqslant S_4$ we require $(q+1)/2=8$ or $6$.  The first is not
  possible while the second has $q=11$ in which case there is no
  $S_4$. To have $A\leqslant \PGL(2,p^{f/2})$ we require
  that $(q+1)/2$ divides either $2(p^{f/2}-1)$ or
  $2(p^{f/2}+1)$. Hence $p^f+1$ divides either $4(p^{f/2}-1)$ or
  $4(p^{f/2}+1)$ and so
$p^{f/2}-1\leq 4$. This implies that $p^{f/2}=3$ or 5. However, we
  then have $q=9$ or 25, and in both cases $(q+1)/2$ is odd. Hence $A$
  is not contained in $\PGL(2,p^{f/2})$. For $A\leqslant \PSL(2,p^m)$,
  for some $m<f/2$, we need $(q+1)/2$ to divides either $p^m-1$ or
  $p^m+1$. Neither of these are possible and so $A$ and $H$ are one of
  the groups listed.  

Suppose next that $E=S_4$ and $q=p\equiv \pm 1\pmod 8$. Then
$A=A_4$. Since $q=p$, the only other subgroup of $G$ containing $A$ is
$H\cong A_5$ when $q\equiv \pm 1\pmod {10}$. Since each $A_5$ contains
5 copies of $A_4$ and the normaliser in $G$ of $A_4$ is $S_4$ it
follows that there are two choices for $H$. This gives case (1).

Finally, if $E=\PGL(2,p^{f/2})$ with $p$ odd then $A=\PSL(2,p^{f/2})$. The only
way that $A$ can be contained in another maximal subgroup of $G$ is if
$A$ is soluble. Hence $q=9$, $A=\PSL(2,3)\cong A_4$. Looking at the
maximal subgroups of $G$ it follows that $H\cong A_5$.  However, in
this case $G$ is 2-transitive on the cosets of $H$, a contradiction.
\end{proof}

We also need the following proposition concerning the exceptional
cases in Theorem \ref{thm:novelties}.

\begin{proposition}
\label{prn:exceptions}
Let $T=\PSL(2,q)\norml G\leqslant \PGL(2,q)$ and suppose that $E$ is a
maximal subgroup of $G$ not containing $T=\PSL(2,q)$ such that 
$E\cap T$ is not maximal in $T$. Suppose that $G$ has a subgroup $H$
such that $A=H\cap E$ is a proper subgroup of $H$ and has index two in
$E$, and that $G$ is not 2-transitive on the set of cosets of
$H$. Then one of the following holds. 
\begin{enumerate} 
\item $G=\PGL(2,7)$, $E=D_{12}$, $H=S_4$ and $A=E\cap T=D_6$.
\item $G=\PGL(2,7)$, $E=D_{16}$, $H=S_4$ and $A=E\cap T=D_8$. 
\item $G=\PGL(2,9)$, $M_{10}$, or $\PGammaL(2,9)$, $E=N_G(D_8)$,
$H=N_G(\PGL(2,3))$ and $A=E\cap\PSigmaL(2,9)$. 
\item $G=\PGL(2,9)$, $M_{10}$, $\PGammaL(2,9)$, $E=N_G(D_{10})$,
$H=N_G(A_5)$ and $A=E\cap \PSigmaL(2,9)$. 
\item $G=\PGL(2,11)$, $E=D_{20}$, $H=C_{11}\rtimes C_{10}$ and
$A=C_{10}$. 
\item $G=\PGL(2,11)$, $E=D_{20}$, $H=A_5$ and $A=E\cap T=D_{10}$. 
\item $G=\PGL(2,q)$, $q=p\equiv \pm 11,\pm 19\pmod{40}$, $E=S_4$,
  $H=A_5$ and $A=A_4$.
\end{enumerate}
In each case there are two conjugate choices for $H$.
\end{proposition}
\begin{proof}
Note that $G$ and $E$ are given by Theorem \ref{thm:novelties}.  The
first 6 cases can all be dealt with by looking at the list of maximal
subgroups in \cite{atlas}. If $G=\PGL(2,q)$ for 
$q=p\equiv 11,19,21,29\pmod{40}$ and $E=S_4$ then the only possibility
for $A$ is $A_4$.  There are then two choices for $H$ being $A_5$ and
these are the only possibilities.
\end{proof}

We can now determine all $G$-edge-primitive graphs with $\soc(G)=\PSL(2,q)$.

\begin{proof}{\bf (of Theorem \ref{thm:PSL})}
Let $\Gamma$ be a $G$-edge-primitive graph such that
$T=\soc(G)=\PSL(2,q)$ with $q>3$. Then by Proposition \ref{prn:general}
there exists a maximal subgroup $E$ of $G$ with an index 2 subgroup
$A$ also contained in a proper corefree subgroup $H$ of $G$ such that
$\Gamma\cong \Cos(G,H,HgH)$ for some $g\in E\backslash A$. 
If $G$ is 2-transitive on the set of cosets of $H$ then $\Gamma$ is a
complete graph and $G$ is primitive on 2-subsets. By Theorem
\ref{thm:prim2subsets}, $G$ appears in Table
\ref{tab:2subsetprim}. Thus we can assume that $G$ is not 2-transitive
on vertices. Then by
Proposition \ref{prn:exceptions} either $\Gamma$ is 
$T$-edge-primitive with $E\cap T$, $A\cap T$ and $H\cap T$
given by Proposition \ref{prn:containment}, or $G,E, A$ and $H$ are given
by Proposition \ref{prn:exceptions}.

Next let $q=p\equiv \pm 1,\pm9\pmod{40}$, $E\cap T=S_4$, $A\cap T=A_4$
and $H\cap T=A_5$.  Since there are two conjugacy classes of $A_5$
subgroups in $\PSL(2,q)$ and these are fused in $\PGL(2,q)$ it follows
that $\PGL(2,q)$ is not an automorphism group of this graph and so we
have row 9 of Table \ref{tab:psl}.

The remaining cases from Proposition \ref{prn:containment} are 
\begin{enumerate}
\item $q=17$, $E\cap T=D_{16}$, $A\cap T=D_{8}$ and $H\cap T=S_4$
\item $q=19$, $E\cap T=D_{20}$, $A\cap T=D_{10}$ and $H\cap T=A_5$ 
\item $q=25$, $E\cap T=D_{24}$, $A\cap T=D_{12}$ and $H\cap T=\PGL(2,5)$. 
\end{enumerate}
In all cases there are two $T$-conjugacy classes of subgroups 
$H\cap T$, and these are fused in $\PGL(2,q)$. Hence we get isomorphic
graphs. Also the only possibilities for $G$ are then
$\PSL(2,17),\PSL(2,19),\PSL(2,25)$ and $\PSigmaL(2,25)$. These give us
rows 6--8 of Table \ref{tab:psl}.

It remains to deal with the groups left from Proposition
\ref{prn:exceptions}.

If $G=\PGL(2,7)$, $E=D_{12}$, $A=D_6$ and $H=S_4$ then since
$H\leqslant \PSL(2,7)$ it follows that $\Gamma$ is bipartite. Note
that $G\cong \Aut(\PSL(3,2))$, $H$ is the stabiliser in $\PSL(3,2)$ of a
1-space $U$  and $A$ is the stabiliser in $H$ of a 2-space which is
complementary to $U$. Thus we have row 2.

Next let $G=\PGL(2,7)$, $E=D_{16}$, $A=D_8$ and $H=S_4$. Again we have
that $\Gamma$ is bipartite, and $H$ is the stabiliser  in $\PSL(3,2)$ of a
1-space $U$. However, this time $A$ is the stabiliser in $H$ of a
2-space containing $U$ and so $\Gamma$ is the Heawood graph, so we
have row 1.

Now let $G=\PGL(2,9)$, $M_{10}$, or $\PGammaL(2,9)$, $E=N_G(D_8)$,
$A=E\cap\PSigmaL(2,9)$ and $H=N_G(\PGL(2,3))$. Note that
$\PGammaL(2,9)\cong\la \PSp(4,2),\tau\ra$ where $\tau$ is a duality of
the associated polar space. Moreover, $H$ is the stabiliser of a
totally isotropic 1-space and $A$ is the stabiliser in $H$ of an
incident totally isotropic 2-space. Thus $\Gamma$ is the Tutte--Coxeter graph
and we have row 4.

When $G=\PGL(2,9)$, $M_{10}$ or $\PGammaL(2,9)$, $E=N_G(D_{10})$,
  $A=E\cap \PSigmaL(2,9)$ and $H=N_G(A_5)$, we have that $H\leqslant
  (G\cap \PSigmaL(2,9))$ and $G\cap \PSigmaL(2,9)$ is an index two
  subgroup of $G$. Thus $\Gamma$ is bipartite. The vertices of
  $\Gamma$ are two sets of size 6 with $\PSL(2,9)\cong A_6$ acting on
  each with two different actions. Since the stabiliser in $A_6$ of a
  point in one action is still transitive in the other action it
  follows that $\Gamma\cong K_{6,6}$ and we have row 3 of Table
  \ref{tab:psl}.

When $G=\PGL(2,11)$, $E=D_{20}$, $A=E\cap T=D_{10}$ and $H=A_5$ we
have that $H\leqslant \PSL(2,11)$ and so we get a bipartite graph on
22 vertices with valency 6. Thus we have row 5.

Finally, let $G=\PGL(2,q)$, $q=p\equiv 11,19,21,29\pmod{40}$, $E=S_4$,
$A=A_4$ and $H=A_5$. Then we get the bipartite graph in row 10 of Table
\ref{tab:psl}. 
\end{proof}

\end{document}